\documentclass[12pt]{amsart}
\usepackage{amsthm,amsfonts,amssymb,amsmath,oldgerm}
\numberwithin{equation}{section}
\usepackage{fullpage}
\usepackage{amsmath}
\usepackage{setspace}
\usepackage{stmaryrd}
\usepackage{mathrsfs}
\usepackage{fancyhdr}
\usepackage{graphicx}
\usepackage{enumerate}
\usepackage{bm}
\usepackage{bbm}
\usepackage{dsfont}
\usepackage{multirow}
\usepackage{psfrag}
\usepackage[font=scriptsize]{caption}
\usepackage[font=scriptsize]{subcaption}
\usepackage{listings}
\usepackage{tikz}
\usepackage{empheq}
\usepackage{cases}
\usetikzlibrary{matrix} 
\usepackage[framemethod=TikZ]{mdframed}
\mdfdefinestyle{MyFrame}{backgroundcolor=gray!20!white}
\usepackage[makeroom]{cancel}

\usepackage[top=20mm,bottom=20mm,left=18mm,right=18mm,a4paper]{geometry}

\usepackage{hyperref,bookmark}

\hypersetup{
    colorlinks,          
    citecolor=blue,        
    filecolor=blue,      
    urlcolor=blue,           
    linkcolor=blue,
}

\newcommand{\ols}[1]{\mskip.5\thinmuskip\overline{\mskip-.5\thinmuskip {#1} \mskip-.5\thinmuskip}\mskip.5\thinmuskip} 

\newcommand\dD{\mathrm{d}}

\def\eps{\varepsilon }

\newcommand{\sign}{{\text{\rm sgn }}}

\newcommand{\mueq}{\cM_{\rho_0/\eps,\cV}\otimes\bar{\mu}_{t,\bx}}

\newcommand\br{\begin{remark}}
\newcommand\er{\end{remark}}
\newcommand\bp{\begin{pmatrix}}
\newcommand\ep{\end{pmatrix}}
\newcommand{\be}{\begin{equation}}
\newcommand{\ee}{\end{equation}}
\newcommand\ba{\begin{equation}\begin{aligned}}
\newcommand\ea{\end{aligned}\end{equation}}
\newcommand\ds{\displaystyle}

\newcommand{\beg}{\begin{example}}
\newcommand{\eeg}{\end{exaplem}}
\newcommand{\bpr}{\begin{proposition}}
\newcommand{\epr}{\end{proposition}}
\newcommand{\bt}{\begin{theorem}}
\newcommand{\et}{\end{theorem}}
\newcommand{\bc}{\begin{corollary}}
\newcommand{\ec}{\end{corollary}}
\newcommand{\bl}{\begin{lemma}}
\newcommand{\el}{\end{lemma}}
\newcommand{\bd}{\begin{definition}}
\newcommand{\ed}{\end{definition}}
\newcommand{\brs}{\begin{remarks}}
\newcommand{\ers}{\end{remarks}}

\newtheorem{theorem}{Theorem}[section]
\newtheorem{proposition}[theorem]{Proposition}
\newtheorem{corollary}[theorem]{Corollary}
\newtheorem{lemma}[theorem]{Lemma}
\newtheorem{remark}[theorem]{Remark}
\newtheorem{definition}[theorem]{Definition}

\newtheorem{example}[theorem]{Example}

\newcommand{\R}{{\mathbb R}}

\newcommand\bx{{\bm x}}

\newcommand{\by}{\bm y}

\newcommand{\bu}{\bm u}

\newcommand\bfu{{\bm u}}

\newcommand\cA{{\mathcal A}}
\newcommand\cB{{\mathcal B}}

\newcommand\cD{{\mathcal D}}
\newcommand\cE{{\mathcal E}}

\newcommand\cH{{\mathcal H}}
\newcommand\cI{{\mathcal I}}
\newcommand\cJ{{\mathcal J}}
\newcommand\cK{{\mathcal K}}
\newcommand\cL{{\mathcal L}}
\newcommand\cM{{\mathcal M}}

\newcommand\cU{{\mathcal U}}
\newcommand\cV{{\mathcal V}}
\newcommand\cW{{\mathcal W}}

\newcommand\scP{{\mathscr P}}
\newcommand\scC{{\mathscr C}}

\usetikzlibrary{fit}
\usetikzlibrary{arrows}
\numberwithin{equation}{section}
\hypersetup{linkbordercolor=red}

\title{Concentration phenomena in FitzHugh-Nagumo's equations: a mesoscopic approach}
\date {\today}

\author{Alain Blaustein and Francis Filbet}

\begin{document}
\maketitle

\centerline{\scshape Alain Blaustein}
{\footnotesize
    \centerline{Institut de Math\'ematiques de Toulouse, Universit\'e Paul Sabatier}
    \centerline{Toulouse, France, \email{alain.blaustein@math.univ-toulouse.fr}}

}

\medskip
\centerline{\scshape Francis Filbet}
{\footnotesize
    \centerline{Institut de Math\'ematiques de Toulouse, Universit\'e Paul Sabatier}
    \centerline{Toulouse, France, \email{francis.filbet@math.univ-toulouse.fr}}

}

\bigskip

\begin{abstract}
We consider a spatially extended mesoscopic FitzHugh-Nagumo model with strong local interactions and prove that its asymptotic limit
converges towards the classical nonlocal reaction-diffusion FitzHugh-Nagumo system.
As the local interactions strongly dominate, the weak solution to the mesoscopic equation under consideration converges to the local equilibrium, which has the form of Dirac distribution concentrated to an averaged membrane potential.

Our approach is based on techniques widely developed in kinetic theory (Wasserstein distance, relative entropy method), where macroscopic quantities of the mesoscopic model are compared with the solution to the  nonlocal reaction-diffusion system. This approach allows to make the rigorous link between microscopic and reaction-diffusion models.
\end{abstract}

\vspace{0.5cm}
\noindent
\textit{Keywords:}
FitzHugh-Nagumo system; Wasserstein distance; blow-up profile.
\\

\noindent
\textit{2010MSC:}
35K57;
35Q92;
35D30.


\section{Introduction}
\label{sec:1}
\subsection{Physical model and motivations}
Neuron models often focus on the dynamics of the electrical potential through the membrane of a nerve cell. These dynamics are driven by ionic exchanges between the neuron and its environment through its cellular membranes. A very precise modeling of these ion exchanges led to the well-known Hodgkin-Huxley model \cite{HH}. In this paper, we shall focus on a simplified version, called the FitzHugh-Nagumo model \cite{FHN1} \cite{FHN2}, which keeps its most valuable aspects and remains relatively simple mathematically. More precisely, the FitzHugh-Nagumo model accounts for the variations of the membrane potential {$v\in\R$} of a neuron coupled to an auxiliary variable {$w\in\R$} called the adaptation variable. It is usually written as follows:
\begin{equation*}
\left\{
    \begin{array}{lll}
        & \displaystyle \dD v_t\,=\,\left(N(v_t)\,-\,w_t \,+\, 
      I_{ext}\right)\,\dD t  \,+\, \sqrt{2}\,\dD B_t\,,
      \\[0.8em]
        & \displaystyle \dD w_t\,=\,A\left(v_t,w_t\right)\,\dD t\,,\\
    \end{array}
\right.
\end{equation*}
where the drift $N$ is a confining non-linearity with the following typical form
\[
N(v)\,=\, v \,-\, v^3\,,
\] 
even though a broader class of drifts $N$ is considered here. On the other hand, the drift $A$ is an affine mapping that has the following form 
\begin{equation*}
A(v,w) \,=\, a\,v \,-\, b\,w \,+\, c\,,
\end{equation*} 
where $a$, $c \in \R$ and $b>0$, which means that $A$  also has some confining properties. Here, the Brownian motion $B_t$ has been added in order to take into account random fluctuations in the dynamics of the membrane potential $v_t$. Another mathematical reason for looking at this system is that it is a prototypical model of excitable kinetics. Interest in such systems stems from the fact that although the kinetics are relatively simple, couplings between neurons can produce complex dynamics, where well-known examples are the propagation of excitatory pulses, spiral waves in two-dimensions, and spatio-temporal chaos. Here, we introduce coupling through the input current $I_{ext}$. More specifically, we consider that neurons interact with one another following Ohm's law and that the conductance between two neurons depends on there spatial location $\bx \in K$, where $K$ is a compact set of $\R^d$. The conductance between two neurons is given by a connectivity kernel $\Phi : K\times K \rightarrow \R$. Hence, in the case of a network composed with $n$ interacting neurons described by the triplet voltage-adaptation-position $(v_i,w_i,\bx_i)_{1\leq i \leq n}$, the current received by neuron $i$ from the other neurons is given by
\[
I_{ext}
\,=\,
-\frac{1}{n}\sum_{j=1}^n \Phi(\bx_i,\bx_j)\,(v^i_t-v^j_t)\,,
\]
where the scaling parameter $n$ is introduced here to re-normalize the contribution of each neuron. According to the former discussion, a neural network of size $n$ is described by the system of equations
\begin{equation*}
\left\{
    \begin{array}{lll}
      & \displaystyle \dD v^i_t\,=\,\left(N(v^i_t)\,-\,w^i_t \,-\,\frac{1}{n}\sum_{j=1}^n \Phi(\bx_i,\bx_j)\,(v^i_t-v^j_t)  \right)\,\dD t  \,+\, \sqrt{2}\,\dD B^i_t\,,
      \\[2em]
        & \displaystyle \dD w^i_t\,=\,A\left(v^i_t,w^i_t\right)\,\dD t\,,
    \end{array}
\right.
\end{equation*}
where $i \in \{1,...,n\}$.  In the formal limit $n\rightarrow + \infty$, the behavior of the latter system may be described by the evolution of a distribution function $f:= f(t,\,\bx,\bu)$, with $\bu=(v,w)\in\R^2$, representing the density of neurons at time $t$, position $\bx\in K$  with a membrane potential $v$ and adaptation variable $w\in \R$. It turns out that the distribution function $f$ solves the following mean-field  equation
\begin{equation*}
        \displaystyle \partial_t f
        \,+\,
        \partial_v 
        \left( 
        \left(N(v)
        \,-\,w
        \,-\, \mathcal{K}_{\Phi}[f]\right)\, f\right) \,+\,
        \partial_w
        \left( 
        A(v,w)f\right) 
        \,-\,
        \partial_v^2 f\,=\,0\,,
\end{equation*}
where the operator $\mathcal{K}_{\Phi}[f]$ takes into account spatial interactions and  is given by
\[
\mathcal{K}_{\Phi}[f](t,\,\bx,v)
\,=\,
\int_{K\times\R^{2}}
\Phi(\bx,\bx')\, (v-v')\,f(t,\,\bx',\bu')
\dD\bx'\,\dD \bu'\,.
\]
See for instance \cite{mlimit1, mlimit2, mlimit3, mlimit4} for more details on the  mean field limit for the FitzHugh-Nagumo system and \cite{mlimit_bolley} for a related model in collective dynamics.\\

Various other types of kinetic models have been derived during the past decades depending on the hypotheses assumed for the dynamics of the emission of an action potential. They include for example integrate-and-fire neural networks \cite{brunel,CCP,CPSS} and time-elapsed neuronal models \cite{PPS,CCDR,CHE1,CHE}.

Let us be more specific on the modeling of interactions between neurons. A common assumption consists in considering that {there are two types} of interactions: strong short range interactions and weak long range interactions (see \cite{assumption Phi}, \cite{HJ quininao/touboul} \& \cite{mlimit4}). Here we consider a connectivity kernel of the following type
\begin{equation}\label{def:phi:eps}
\Phi^\eps(\bx,\bx')
\,=\,
\frac{1}{\eps}
\delta_0(\bx-\bx')\,+\, \Psi(\bx,\bx')\,,
\end{equation}
where the Dirac mass $\delta_0$ accounts for strong short range interactions with strength $\eps>0$, whereas the connectivity kernel $\Psi : K \times K \rightarrow \R$ is more regular and represents weak long range interactions.

The purpose of this article is to go through the mathematical analysis of the neural network in the regime of strong interactions, that is when $\eps \ll 1$. More precisely, we prove that the voltage distribution concentrates to a Dirac mass by providing a comprehensive description of this concentration phenomenon. 

\subsection{Formal derivation}
Our problem is multiscale due to interactions between neurons, which induce macroscopic effects at the mesoscopic level.
Consequently, we introduce integrated quantities. First, we consider the spatial distribution of neurons throughout the network
\begin{equation*}
\rho_0^\eps(\bx)
=
\int_{\R^2}f^\eps(t,\,\bx,\bu)\,\dD \bu\,.
\end{equation*}
It is straightforward to check that $\rho_0^\eps$ is indeed
time-homogeneous, integrating the mean field equation with respect to
$\bu\in\R^2$. Second, we introduce the averaged  voltage $\cV^\eps$
and adaptation variable $\cW^\eps$  at a spatial location $\bx$
\begin{equation}\label{macro:q}
\left\{
    \begin{array}{ll}
        \displaystyle \rho_0^\eps
        \left(
        \bx
        \right)\,\cV^\eps(t,\,\bx) 
        &\,=\,
        \ds\int_{\R^2}v~f^\eps(t,\,\bx,\bu)\,\dD\bu\,,
        \\[1.1em]
        \displaystyle \rho_0^\eps
        \left(
        \bx
        \right)\,\cW^\eps(t,\,\bx) 
        &\,=\,
        \ds\int_{\R^2}w~f^\eps(t,\,\bx,\bu)\,\dD\bu\,.
    \end{array}
\right.
\end{equation}
In the sequel, we use the vector notation 
$\ds \mathcal{U}^\eps 
\,=\, 
\left(\,
\mathcal{V}^\eps ,
\mathcal{W}^\eps\, 
\right)
$. 
At the mesoscopic level, we compare probability density functions using the Wasserstein distances. Hence, we renormalize $f^\eps$ as
\[
\rho_0^\eps \,
\mu^\eps
\,=\,
f^\eps\,,
\]
where $\mu^\eps$ is a non-negative function which lies in 
$
\ds
\scC^0
\left(
\R^+
\times 
K\,,\,
L^1
\left(
\R^2
\right)
\right)
$ and verifies 
\[
\int_{\R^2}
\,
\mu^\eps(t,\,\bx,\bu)\,\dD \bu
\,=\,
1
,~~
\forall
\left(
t,\,\bx
\right)
\,\in\,
\R^+
\times 
K\,.\]
Consequently, we denote 
$
\ds
\mu^\eps_{t,\,\bx}
$
the probability density function defined as
$
\ds
\mu^\eps_{t,\,\bx}
\,=\,
\mu^\eps
\left(
t,\,\bx,\cdot
\right)
$.\\
With these notations and our modeling assumptions on the connectivity kernel $\Phi^\eps$ defined by \eqref{def:phi:eps}, the mean-field equation rewrites
\begin{equation}
  \label{kinetic:eq}
  \ds\partial_t \, \mu^\eps
        \,+\,
        \mathrm{div}_{\bu}
        \left[ \,
        \mathbf{b}^\eps
         \mu^\eps\,
         \right] 
       \,-\,
       \partial^2_v \,
        \mu^\eps
        \,=\,
        \frac{1}{\eps}\,
        \rho^\eps_0\,
        \partial_v 
        \left[\,
        (v-\cV^\eps)
        \,\mu^\eps\,
        \right]
        ,
\end{equation}
where $\mathbf{b}^\eps$ is defined for all 
$
\ds
\left(
t,\,\bx,\bu
\right)
\in
\R^+
\times
K
\times 
\R^2
$
as
\begin{equation*}
\mathbf{b}^\eps(t,\,\bx,\bu)
\,=\,
\begin{pmatrix}
\ds\,
N(v)
\,-\,w \,-\,
\mathcal{K}_{\Psi}
[\rho_0^\eps \,\mu^\eps]
\left(
t,\,\bx,v
\right)\,\\[0,9em]
\ds
A(\bu)
\end{pmatrix}
\,.
\end{equation*}
Furthermore, one can notice that the non local term $\ds\mathcal{K}_{\Psi}
        [\rho_0^\eps \,\mu^\eps]$ can be expressed in terms of the macroscopic quantities 
\[
\mathcal{K}_{\Psi}
        [\rho_0^\eps \,\mu^\eps](t,\,\bx,v)
\,=\,
\Psi*_r\rho^\eps_0(\bx) \,v \,-\, \Psi*_r(\rho^\eps_0 \,\cV^\eps)(t,\,\bx)\,,
\]
where $*_r$ is a shorthand for the convolution on the right side of any function $g$ with $\Psi$
\[
\Psi*_r g(\bx)
\,=\,
\int_{K}
\Psi(\bx,\bx')\,g(\bx')\,\dD \bx'\,.
\]
Coming back to the analysis of the strong interaction regime, we look for the leading order in \eqref{kinetic:eq}. In our case, it is induced by strong short range interactions between neurons, and as $\eps\rightarrow 0$, we expect
$$
( v \,-\, \cV^\eps)\,
\mu^\eps\,=\, 0\,,
$$
which means that $\mu^\eps$ concentrates around its mean value with respect to the voltage variable at each spatial location $\bx$ in $K$, that is, $\mu^\eps$ converges to a Dirac mass centred in $\cV^\eps$. Thus, to quantify the asymptotic behavior of $\mu^\eps$ when $\eps\ll 1$,  we denote by $\scP_2(\R^2)$ the set of probability laws with finite second order moments
\[
\scP_2\left(\R^2\right)
\,=\,
\left\{
\mu \in \scP(\R^2), \quad
\int_{\R^2}
|\bu|^2
\dD\mu(\bu) \,<\,+\infty
\right\}
\]
and the Wasserstein distance of order two $W_2$ defined as follows: for any $\mu$ and $\nu$ probability measures in $\scP_2(\R^2)$,
\[
W^2_2\left(\mu,\nu\right)
\,=\,
\inf_{\pi \in \Pi(\mu,\nu)}
\int_{\R^4}
|\bu-\bu'|^2
\dD\pi(\bu,\bu')\,,
\]
where  $\Pi(\mu,\nu)$ stands for the set of distributions $\pi$ over $\R^4$ with marginals $\mu$ with respect to $\bu$ and $\nu$ with respect to $\bu'$, that is, for $\pi\in \Pi(\mu,\nu)$
$$
\pi
\left(\,
A,\,\R^2\,
\right) 
= 
\mu(A) \quad{\rm and}\quad \pi
\left(\,\R^2,\,A\,
\right) = \nu(A)\,, 
$$
for any Borel set $A\subset \R^2$.
{The choice of the $W_2$ metric in our analysis is somehow arbitrary as it might be possible to adapt our approach to other distances metrizing probability spaces. However,} it is easy to quantify the distance
between $\mu^\eps$ and a Dirac mass in this framework since we have

$$
W_2^2\left(\,\mu_{t,\,\bx}^\eps,\, \delta_{\cV^\eps} \otimes\, \bar{\mu}_{t,\,\bx}^\eps\,\right) \,=\, 
\int_{\R^2} \left|v-\cV^\eps \right|^2
\mu^\eps_{t,\,\bx}(\bu) \dD \bu\,,
$$
where $\bar{\mu}^\eps$ is defined as the marginal of $\mu^\eps$ with respect to the voltage variable
\begin{equation*}
\bar{\mu}^\eps_{t,\,\bx}(w)
\,=\,
\int_{\R}\mu^\eps_{t,\,\bx}(v,w)\,\dD v\,.
\end{equation*}
Hence, we  multiply equation \eqref{kinetic:eq} by $|v-\cV^\eps|^2$
and integrate with respect to  $\bu\in\R^2$. We obtain that at each
spatial location $\bx$ in $K$, we have {(see Proposition \ref{prop:2} below for a more precise estimate)}
\[
W_2\left(\,\mu_{t,\,\bx}^\eps,\, \delta_{\cV^\eps(t,\,\bx)} \otimes\, \bar{\mu}_{t,\,\bx}^\eps\,\right) \underset{\eps \rightarrow 0}{\sim}
\sqrt{\eps}\,.
\]
Considering this estimate, we infer that the dynamics of the network when $\eps \ll 1$ are driven by the couple 
$\ds(\cV^\eps,\bar{\mu}^\eps)$,
which displays both the
macroscopic \& the mesoscopic scale. We complete this step of our analysis by deriving the limit of $\ds(\cV^\eps,\bar{\mu}^\eps)$.
Multiplying equation \eqref{kinetic:eq} by $v$~(resp.~$1$) and then integrating over $\bu\in\R^2$~(resp.~$v\in\R$), we obtain that the couple $\ds(\cV^\eps,\bar{\mu}^\eps)$ solves the following system
\begin{equation}\label{macro-eps:eq}
    \left\{
    \begin{array}{llll}
        &\displaystyle \partial_t\cV^\eps \,=\, 
        N(\cV^\eps)
        \,-\,
        \cW^\eps
        \,-\,
       \mathcal{L}_{\rho^\eps_0}[\mathcal{V^\eps}]
        \,+\,
        \cE(\mu^\eps)\,
        ,\\[0.8em]
        &\displaystyle \partial_t \bar{\mu}^\eps 
        \,+\,
        \partial_w
        \left(
        a\int_{\R} v \mu^\eps \dD v 
        \,-\,b \,w\, \bar{\mu}^\eps 
        \,+\,
        c\, \bar{\mu}^\eps 
        \right) = 0\,,
  \end{array}
\right.
\end{equation}
with
\[
\cW^\eps
        \,=\,
        \int_{\R}w \,\bar{\mu}^\eps \, \dD w\,.
        \]
In equation \eqref{macro-eps:eq},  $\mathcal{L}_{\rho_0^\eps}[\cV^\eps]$ is a non local operator given by
\[
\mathcal{L}_{\rho_0^\eps}[\cV^\eps]
=
\cV^\eps \,\Psi*_r\rho_0^\eps \,-\, \Psi*_r(\rho_0^\eps \cV^\eps)\,,
\]
and  the error term $\cE(\mu^\eps)$ is given by
\begin{equation}\label{error}
\cE(\mu^\eps)
\,=\,
\,\int_{\R^2} N(v)\,
        \mu^\eps_{t,\,\bx}(\bu)\, \dD \bu\,-\, N(\cV^\eps)\,.
\end{equation}
Before computing the limit for 
$\ds(\cV^\eps,\bar{\mu}^\eps)$,
we emphasize that multiplying the second equation in \eqref{macro-eps:eq} by $w$ and  integrating with respect to $w\in\R$, we get a closed equation for $\cW^\eps$ since  $A$ is affine, 
\[
\partial_t \cW^\eps
\,=\, A\left(\cV^\eps, \cW^\eps
\right)\,.
\]
Both equations on $\cV^\eps$ and $\bar{\mu}^\eps$ in \eqref{macro-eps:eq} depend on the distribution function $\mu^\eps$. However since our interest here lies in the regime of strong interactions, we replace $\mu^\eps$ in \eqref{macro-eps:eq} by the \textit{ansatz} 
\[
\ds
\mu^\eps
\,\underset{\eps \rightarrow 0}{=}
\,
\delta_{\cV^\eps}
\otimes \bar{\mu}^\eps
\,+\,
O(\sqrt{\eps})\,.
\]
This removes the dependence with respect to $\mu^\eps$ from the system \eqref{macro-eps:eq}. Indeed, we obtain on the one hand
(see Proposition \ref{estimate for the error 2} for more details)
\[
(\cV^\eps,\, \cW^\eps)\,
\underset{\eps \rightarrow 0}{=}\,\, (\cV,\,\cW) \,\,+\,\, O\left(\eps
\right)\,, \quad{\rm as}\quad \eps\rightarrow 0\,,
\]
and on the other hand
\[
\bar{\mu}^\eps
\underset{\eps \rightarrow 0}{=}\,\, 
\bar{\mu} \,\,+\,\, O\left(\sqrt{\eps}
\right)\,, \quad{\rm as}\quad \eps\rightarrow 0\,,
\]
where the couple $(\cV,\bar{\mu})$ solves the following system
\begin{equation}
  \label{macro:eq}
    \left\{
    \begin{array}{l}
        \displaystyle \partial_t\cV
        \,=\, 
        N(\cV)
        \,-\,
        \cW
        \,-\,
       \mathcal{L}_{\rho_0}[\cV]
        \,,\\[0.9em]
        \displaystyle \partial_t \bar{\mu}
        \,+\,
        \partial_w
        \left(
        A(\cV\,,\,w)\, \bar{\mu}
        \right) \,=\, 0\,,
    \end{array}
\right.
\end{equation}
with
$$
 \cW
        \,=\,
        \int_{\R}w\,\dD\bar{\mu}_{t,\,\bx}(w)\,.
        $$
Similar results have already been obtained in a deterministic setting in \cite{limite-hydro-crevat} using relative entropy methods. In the end,  it can be proven that $\mu^\eps$ converges to a mono-kinetic distribution in $v$ with mean $\cV$ and we get the following rate of convergence {(see Corollary \ref{order 0} for more details)}
\begin{equation*}
W_2\left(\mu^\eps_{t,\,\bx},\, \delta_{\cV(t,\,\bx)}
\otimes \bar{\mu}_{t,\,\bx}\right)
\underset{\eps \rightarrow 0}{=}
O(\sqrt{\eps})\,.
\end{equation*}
Actually, this latter convergence estimate corresponds to the expansion at order 0 of $f^\eps$ in the regime of strong interactions. 

\subsection{Introduction of rescaled variables}
In this paper, we introduce some rescaled variables, in order to study
more precisely the asymptotic behavior of the solution and  to improve the order of convergence. The strategy consists in finding the concentration's profile with respect to the potential variable $v$.  This leads to considering the following re-scaled version $\nu^\eps$ of $\mu^\eps$
\begin{equation*}
    \mu^\eps(t,\,\bx,v,w) 
    \,=\,
    \frac{1}{\eps^\alpha}\,
    \nu^\eps 
    \left(t,\,\bx,\,\frac{v-\cV^\eps}{\eps^\alpha},\,
    w-\cW^\eps
    \right)\,,
\end{equation*}
where $\eps^\alpha$ is the concentration rate of $\mu^\eps$ around its mean value $\cV^\eps$ and $\alpha$ needs to be determined. For a proper choice of $\alpha$, we expect $\nu^\eps$ to converge to some limit as $\eps$ vanishes. This limit will be interpreted as the concentration profile of the voltage's distribution throughout the network in the regime of strong interaction.

In order to determine the proper concentration exponent $\alpha$, we derive the equation solved~by~$\nu^\eps$. To this aim, we perform the following change of variable
\begin{equation}
  \label{change:var}
(v,\,w)\mapsto \left(\frac{v-\cV^\eps}{\eps^\alpha},\,w-\cW^\eps\right)
\end{equation}
in equation  \eqref{kinetic:eq} and use the first equation of \eqref{macro-eps:eq} on $\cV^\eps$. It yields
\begin{equation*}
\partial_t\, \nu^\eps
\,+\,
\mathrm{div}_{\bu}
\left[
\,\mathbf{b}^\eps_0\,\nu^\eps \right]
\,=\,
\frac{1}{\eps^{2\alpha}}
\partial_v
\left[\eps^{2\alpha-1}\,\rho_0^\eps
\,v\, \nu^\eps
+
\partial_v \,\nu^\eps
\right]\,,
\end{equation*}
where $\mathbf{b}^\eps_0$ is a centered version of $\mathbf{b}^\eps$ and is given by
\begin{align*}
\ds
\mathbf{b}^\eps_0
\left(
t,\,\bx,\,\bu
\right)
\,=\,
\begin{pmatrix}
\ds\,
\eps^{-\alpha}
\left(
\,
N(\cV^\eps
\,+\,
\eps^\alpha v)
\,-\,
N(\cV^\eps) 
\,-\,
w
\,-\,
\eps^\alpha v\, \Psi *_r \rho^\eps_0(\bx)
\,-\,
\mathcal{E}
\left(
\mu^\eps
\right)
\right)
\\[0,9em]
\ds
A_0(\eps^\alpha v,\, w)
\end{pmatrix}\,,
\end{align*}
 and where $A_0$ is the linear version of $A$
\begin{equation*}
A_0(\bu) = A(\bu)-A(\mathbf{0})\,.
\end{equation*}
It turns out that the only suitable value for $\alpha$ is $1/2$. Indeed, when $\ds \alpha$ is less than $1/2$, we check that  $\nu^\eps$ converges towards a Dirac mass, which means that the scaling is not precise enough. On the contrary, when
$\ds \alpha \,>\, 1/2$,  $\nu^\eps$ converges to $0$, which means that we "zoom in" too much. Hence we obtain the following equation 
\begin{equation}\label{nu:eq}
\partial_t\, \nu^\eps
\,+\,
\mathrm{div}_{\bu}
\left[\,
\mathbf{b}^\eps_0
\,\nu^\eps \,\right]
\,=\,
\frac{1}{\eps}\,
\partial_v
\left[\,\rho_0^\eps
\,v\, \nu^\eps
+
\partial_v \,\nu^\eps\,
\right]\,,
\end{equation}
where we take $\alpha=1/2$ in the definition of $\mathbf{b}^\eps_0$.
Keeping only the leading order, it yields that
\[
\nu^\eps_{t,\,\bx}(v,w) \underset{\eps \rightarrow 0}{=}
\cM_{\rho_0^\eps(\bx)}(v)\otimes \bar{\nu}^\eps_{t,\,\bx}(w) \,+\, O(\sqrt\eps)\,,
\]
where the Maxwellian $\cM_{\rho_0^\eps}$
is defined as
\begin{equation*}
    \mathcal{M}_{\rho_0^\eps }(v)=
    \sqrt{\frac{\rho_0^\eps}{2\pi}}\exp
    \left(
    -\frac{\rho_0^\eps \,|v|^2}{2}
    \right)\,,
\end{equation*}
whereas $\bar{\nu}^\eps$ is the marginal of $\nu^\eps$ with respect to the re-scaled adaptation variable
\begin{equation*}
\bar{\nu}^\eps_{t,\,\bx}(w)
\,=\,
\int_{\R}\nu^\eps_{t,\,\bx}(\bu)\,\dD v\,.
\end{equation*}
At this point, it is possible to answer our initial concern:
$\mu^\eps$ concentrates with Gaussian profile as $\ds\eps \rightarrow 0$. 
Then we complete the analysis by deriving the limit of $\bar{\nu}^\eps$ :  integrating  equation \eqref{nu:eq} with respect to the re-scaled voltage variable $v$, we obtain the equations solved by $\bar{\nu}^\eps$,
\begin{equation*}
\displaystyle \partial_t \bar{\nu}^\eps
        \,-\,b\,
        \partial_{w}
        \left( 
        w\, \bar{\nu}^\eps
        \right)\,=\,
        -a\,\sqrt{\eps}\,\partial_{w}\int_\R v\, \nu_{t,\,\bx}^\eps(\dD v, w)\,.
\end{equation*} 
Once again, the equation still depends on $\nu^\eps$ through the source term in the right-hand side. 
However we obtain that it is in fact of order $\eps$ when we replace $\nu^\eps$ with the following \textit{ansatz} 
\[
\nu^\eps 
\,
 \underset{\eps \rightarrow 0}{=}\,
\cM_{\rho_0^\eps}\otimes \bar{\nu}^\eps \,+\, O(\sqrt\eps)\,.\]
Consequently, we expect the following convergence
\begin{equation*}
  \bar{\nu}^\eps \,\underset{\eps \rightarrow 0}{=}\,
  \bar{\nu} \,  + O(\eps)
  \,,
\end{equation*}
where $\bar\nu$ solves the following linear transport equation
\begin{equation*}
\partial_t \bar\nu
        \,-\,b\,
        \partial_{w}
        \left( 
        w \,\bar\nu\right) \,=\, 0\,,
\end{equation*}
which corresponds to the same equation as \eqref{macro:eq} for $\bar\mu$ after inverting the change of variable \eqref{change:var}. \\ 
We come back to our initial problem, which consisted in building a precise model for the dynamics of $\mu^\eps$ in the regime of strong interactions. We invert the change of variable \eqref{change:var} and obtain in the end
\begin{equation*}
\ds
W_2\left(
\mu^\eps_{t,\,\bx},\,
\mathcal{M}_{\,
\frac{1}{\eps}\,\rho_0(\bx)\,,\,\mathcal{V}(t,\,\bx)}
\otimes
\bar{\mu}_{t,\,\bx}
\right)
\,
\underset{\eps \rightarrow 0}{=}\,
O(\eps)\,,
\end{equation*}
where 
$
\ds
\mathcal{M}_{
\rho_0/\eps,\,\mathcal{V}}
$
is given by 
$
\ds
\mathcal{M}_{
\rho_0/\eps,\,\mathcal{V}}\left(v\right)
\,=\,
\mathcal{M}_{
\rho_0/\eps}
\left(v-\mathcal{V}\right)
$. This result should be regarded as the expansion of $\mu^\eps$ at order $1$  in the regime of strong interactions. It may be compared with the expansion of $\mu^\eps$ at order $0$. 
Furthermore, it enables us to characterize the blow up profile of the distribution function $\mu^\eps$ and to improve the order of convergence.\\
Our result is in line with a broader collection of publications, which focus on the mathematical analysis of the dynamics in a FitzHugh-Nagumo neural networks with strongly interacting neurons. First, we mention \cite{HJ quininao/touboul}, in which similar results are obtained following a Hamilton Jacobi approach. 
The authors study the so-called Hopf-Cole transform $\phi^\eps$ of $\mu^\eps$ defined by the following  \textit{ansatz}
\[
\mu^\eps\,=\,
\exp{
\left(
\phi^\eps / \eps
\right)
}\,.
\]
However, due to this \textit{ansatz}, authors deal with well prepared initial condition in the sense that it is already concentrated at time $t=0$. In our case, we lift this assumption and deal with non-concentrated initial data. Furthermore, the results are stated in a spatially homogeneous setting and the limiting distribution $\bar{\mu}$ of the adaptation variable is not identified in \cite{HJ quininao/touboul}. Secondly,
our work follows on from \cite{limite-hydro-crevat}, which focuses on the expansion of $\mu^\eps$ at order $0$ in a deterministic setting. On top of that, we cite \cite{mlimit_bolley} which deals with mean-field limit in the context of collective dynamics. This article locates itself in a probability framework and the authors develop mathematical methods based on the Wasserstein distance and similar to the ones in the present article. However, the authors adopt a stochastic point of view whereas we focus on the analytic point of view all along this paper. To end with, we mention \cite{fournier/perthame}, where methods related to Wasserstein distances are reviewed for a broad class of models.

The rest of the paper consists in making the asymptotic expansion rigorous. In the next section, we state our assumptions on the parameters of our problem: $N$, $\Psi$ and 
$
\left(
\mu^\eps_{0}
\right)_{\eps\,>\,0}
$. Then we state the main result, Theorem \ref{main:th}. Section
\ref{sec:3} is devoted to {\it a priori} estimates on the solutions
$(\mu^\eps)_{\eps\,>\,0}$, whereas Section \ref{sec:4} contains the proof
of Theorem \ref{main:th}. Finally, in the Appendix, we give the main
ingredients to prove existence and uniqueness of a solution $\mu^\eps$
to equation \eqref{kinetic:eq} for any $\eps>0$. We mention that even
though this well posedness result is not our main concern here, we develop interesting arguments using a modified relative entropy which might have other applications.

\section{Mathematical setting \& main results}
\label{sec:2}

In this section, we give the precise mathematical setting of our analysis and present our main results on the profile of the distribution function $f^\eps$ when $\eps \ll 1$.

\subsection{Mathematical setting}
\label{sec:21}
We suppose the drift $N$ to be of class $\scC^2$ over $\R$. Then we set 
$
\ds
\omega(v)
\,=\,
N(v)/v
$ and suppose that the following coupled pair of confining assumptions are met
\begin{subequations}
\begin{numcases}{}
\label{hyp1:N}
\ds
    \limsup_{|v|\,\rightarrow\,+ \infty}\,
    \omega(v)
    \,=\,-\infty\,,\\[0,9em]
\label{hyp2:N}
\ds
    \sup_{|v|\,\geq\,1}\,
    \frac{
    \left|\omega(v)\right|
    }{|v|^{p-1}}
    \,<\, + \infty\,,
\end{numcases}
\end{subequations}
for some $p \geq 2$. Assumption \eqref{hyp1:N} ensures that $N$ is {super-linearly confining} at infinity. It allows us to obtain uniform estimates in time. However, it would have been replaced by 
\[
\ds
\sup_{|v|\,\geq\, 1}
\omega(v)\,<\,+\infty\,,\]
had we worked on finite time intervals.
Assumption \eqref{hyp2:N}
ensures that the {confinement property} of $N$ is controlled by a polynomial. This assumption is merely technical. Indeed, it may be possible to consider exponential control instead, as long as we also suppose exponential moments for the initial condition $\mu^\eps_0$. In the end, the assumption only dictates the localization we need on the initial condition: 
in our case, we choose polynomial control on the drift $N$ and polynomial moments on the initial condition. A typical choice for $N$ would be any polynomial of the form
\[
P(v)
\,=\,
Q(v)\,
-\,Cv|v|^{p-1}\,,
\]
for some positive constant $C\,>\,0$ and where $Q$ has degree less then $p$.

We turn to the connectivity kernel $\Psi$. We suppose
$
\Psi
\,
\in
\,
\scC^0
\left(
K_{\bx},\,
L^1
\left(
K_{\bx'}
\right)
\right)
$ and assume the following bound to hold
\begin{equation}
\label{hyp1:psi}
\sup_{\bx'\in K} \int_{K}\left|\Psi(\bx,\bx')\right| \,\dD\bx \,<\, +\infty\,.
\end{equation}
Moreover, we consider $r$ in $\ds ]1,\,+\infty]$, define its conjugate $r'\geq 1$ as $1/r + 1/r' = 1$ and suppose
\begin{equation}
\label{hyp2:psi}
\sup_{\bx\in K} \int_{K}\left|\Psi(\bx,\bx')\right|^r \,\dD\bx' \,<\, +\infty\,.
\end{equation}
Our set of assumptions on the connectivity kernel is quite general
since  we consider non-symmetric interactions between neurons and also
authorize the connectivity kernel to follow a {negative} power law, a case which is considered in the physical literature  (see \cite{mlimit4}).
\begin{remark}
It may be possible to adapt our analysis to the case of a discrete spatial variable. This could be done by replacing the Lebesgue measure $\dD \bx$ in the definition of $\Psi$ by a positive Borel measure $\lambda_K$ with finite mass over $K$ (typically a sum of Dirac mass if $K$ is discrete). In this case, we should  also suppose that $\rho_0^\eps$ lies in 
$\ds
L^1
\left(
\lambda_K
\right)
$.
\end{remark}
We now state our assumptions on the sequence of initial data 
$
\left(
\mu^\eps_{0,\bx}
\right)_{\eps\,>\,0}
$. We suppose that for each $\eps\,>\,0$ we have 
\[
\ds
\left(
\bx
\,
\mapsto\,
\mu^\eps_{0,\bx}
\right)
\in
\scC^0
\left(
K\,,\,
\scP_2
\left(\R^2
\right)
\right)\,.
\]
{We also suppose the spatial distribution of the network $\rho_0^\eps$ to be continuous over $K$ and uniformly bounded from above and below, that is, there exists a constant $m_*\,>\,0$ such that for all $\eps > 0$
\begin{equation}
\label{hyp:rho0}
m_* \leq \rho_0^\varepsilon \leq 1/m_*\,,\quad\rho_0^\eps \in \scC^0
\left(
K
\right)\,.
\end{equation}}
On top of that, we assume the following condition: there exist two positive constants $m_p$ and $\ols{m}_p$, independent of $\eps$, such that
\begin{equation}
\label{hyp1:f0}
\sup_{\bx \in K}
\,\int_{\R^2}
|\bu|^{2p} \mu^\eps_{0,\bx} (\dD \bu)
\,\leq\, m_p\,,
\end{equation}
and such that
\begin{equation}
\label{hyp2:f0}
\int_{K \times \, \R^2}\,
|\bu|^{2p r'}
\,\rho_0^\eps(\bx)\,\mu^\eps_{0,\bx} 
(
\dD \bu)\,\dD \bx
\,\leq\, \ols{m}_p\,,
\end{equation}
where $p$ and $r'$ are given in \eqref{hyp2:N} and \eqref{hyp2:psi}.\\

Now let us define the notion of solution we will consider for equation \eqref{kinetic:eq}.
\begin{definition}\label{notion de solution}
For all $\varepsilon>0$ we say that $\mu^\varepsilon$ solves \eqref{kinetic:eq} with initial condition $\mu^\varepsilon_0$ if we have
\begin{enumerate}
\item  $\mu^\eps$ lies in
\[ 
\scC^0
\left(
\R^+\times K\,,\,
L^1
\left(\R^2\right)
\right)
\cap 
L^\infty_{loc} \left(
\R^+\times K\,,\,
\scP_2
\left(\R^2\right)
\right)\,
,\] 
\item for all $\bx \in K$, $t \geq 0$,
and $\ds\varphi \in \scC_c^\infty
\left(\R^2\right)$
, it holds
\begin{align*}
\int_{\R^2}
\varphi(\bu)\,
\left(
\mu^\varepsilon_{t,\bx}
-
\mu^\varepsilon_{0,\bx}
\right)(\bu)\,\dD\bu\,
&=\,
\int_0^t
\int_{\R^2}
\left(
\nabla_{\bu} \varphi(\bu)
\cdot 
\mathbf{b}^\eps(s,\bx,\bu)
        \,+\,
\partial_v^2
 \varphi(\bu)
\right)
\mu^\varepsilon_{s,\bx}(\bu)\,\dD\bu\,\dD s\\[0,8em]
&-\frac{\rho_0^\eps(\bx)}{\eps}
\int_0^t
\int_{\R^2}
\partial_v
 \varphi(\bu)\,
        \left(v-\cV^\eps(s,\bx)\right)
\mu^\varepsilon_{s,\bx}(\bu)\,\dD\bu\,\dD s\,,
\end{align*}
\end{enumerate}
where $\mathcal{V}^\eps$ and $\mathbf{b}^\eps$  are given by \eqref{macro:q} and \eqref{kinetic:eq}.
\end{definition}
With this notion of solution, equation \eqref{kinetic:eq} is well-posed. Indeed, we have
\begin{theorem}
\label{WP mean field eq}
For any $\varepsilon > 0$, suppose that assumptions \eqref{hyp1:N}-\eqref{hyp2:N} and \eqref{hyp1:psi}-\eqref{hyp:rho0} are fulfilled and that the initial condition $\mu_0^\varepsilon$ also verifies
\begin{equation}
\label{hyp3:f0}
\left\{
\begin{array}{l}
\ds\sup_{\bx \in K}
    \int_{\R^2}
    e^{|\bu|^2/2}\,
    \mu^\eps_{0,\bx}(\dD \bu)\,
    \leq\, M^\eps\,,
    \\[1.1em]
    \ds \sup_{\bx \in K}
    \int_{\R^2}
    \ln(\mu^\eps_{0,\bx})\,\mu^\eps_{0,\bx}(\dD \bu)\,
    \leq\, m^\eps\,,
    \end{array}\right.
\end{equation}
and
\begin{equation}\label{hyp4:f0}
\sup_{\bx \in K}
    \left\|
\nabla_{\bu}
\sqrt{\mu^\eps_{0,\bx}}\,
\right\|^2_{L^2(\R^2)}\,
    \leq\, m^\eps\,,
\end{equation}
where $M^\eps$ and $m^\eps$ are two positive constant. Then there exists a unique solution $\mu^\varepsilon$ to equation \eqref{kinetic:eq} with initial condition $\mu^\varepsilon_0$, in the sense of Definition \ref{notion de solution} which verifies
\[
\displaystyle
\sup_{(t\,,\,\bx) \in [0,T]\times K}\,
\int_{\R^2}
e^{|\bu|^2/2}\, \mu^\varepsilon_{t,\bx}(\bu)\,
\dD \bu < +\infty\,,
\]
for all $T\,\geq\,0$.\\ Furthermore, the macroscopic quantities 
$
\ds
\mathcal{V}^\eps
$ and 
$
\ds
\mathcal{W}^\eps
$ given in \eqref{macro:q} lie in
$
\ds
\scC^0
\left(
\R^+ \times K 
\right)
$.
\end{theorem}

We postpone the proof of this result to the Appendix \ref{Appendix}, which relies on relative entropy estimates. We take advantage of assumption \eqref{hyp3:f0} to derive continuity estimates for both the time and the spatial variable. More precisely, we apply an abstract result from \cite{Bolley/Villani} which ensures that if we suppose some exponential moments such as in \eqref{hyp3:f0}, then the Wasserstein metric is controlled by the relative entropy. We make use of assumption \eqref{hyp4:f0} in order to obtain strong continuity with respect to the time variable. We emphasize that since we are not able to close the estimates using directly the $L^1$ norm, we introduce some modified relative entropy, which turns out to be, in some sense, equivalent to the $L^1$ distance. Our approach is original to our knowledge.  

\begin{remark} We do not make use of assumptions \eqref{hyp3:f0} \& \eqref{hyp4:f0} in the analysis of the asymptotic $\eps \rightarrow 0$. Therefore, both constants $M^\eps$ and $m^\eps$ may blow up as $\eps$ vanishes. 
\end{remark}
We also define solutions for the limiting system \eqref{macro:eq}
\begin{definition}\label{sol:macro}
We say that $
\ds
\left(
\mathcal{V}\,,\,
\bar{\mu}
\right)
$ solves \eqref{macro:eq} with initial condition 
$
\ds
\left(
\mathcal{V}_0\,,\,
\bar{\mu}_0
\right)
$
if we have
\begin{enumerate}
    \item 
    $
    \ds
    \mathcal{V} \in \scC^0
    \left(
    \R^+\times K
    \right)
    $
    and 
    $
    \ds
    \bar{\mu} \in \scC^0
    \left(
    \R^+\times K\,,\,
    L^1
    \left(
    \R
    \right)
    \right)
    $,
    \item $\mathcal{V}$ is a mild solution to \eqref{macro:eq},
    \item for all $t\,\geq\,0$, all $\bx \in K$ and all 
    $
    \ds
    \phi
    \in 
    \scC^\infty_c(\R)
    $ we have
    \[
    \int_{\R}
    \phi(w)\,
    \bar{\mu}_{t,\bx}(w)\,\dD w
    \,=\,
    \int_{\R}
    \phi(w)\,
    \bar{\mu}_{0,\bx}(w)\,\dD w
    \,+\,
    \int_0^{t}
    \int_{\R}
        A(\cV,w)\,
        \partial_w\,
    \phi(w)\,
    \bar{\mu}_{s,\bx}(w)\,\dD w\,\dD s\,.
    \]
\end{enumerate}
\end{definition}
\begin{theorem}\label{wp macro eq}
Under assumptions \eqref{hyp1:N}, \eqref{hyp2:psi}-\eqref{hyp:rho0}, and for any initial condition
\[
\left(
\mathcal{V}_0\,,\,
\bar{\mu}_0
\right)
\in 
\scC^0
\left(
K
\right)
\times
\scC^0
\left(
K\,,\,
L^{1}
\left(
\R
\right)
\right)
,\]
there exists a unique  solution to \eqref{macro:eq} in the sense of Definition \ref{sol:macro} with initial condition
$
\ds
\left(
\mathcal{V}_0\,,\,
\bar{\mu}_0
\right)
$.\\
Furthermore, $\mathcal{V}$ is uniformly bounded over $\R^+ \times K$.
\end{theorem}
\begin{proof}
The key argument is that the system \eqref{macro:eq} may be decoupled through the change of variable 
\eqref{change:var}.
Indeed, we consider the following system
\begin{equation*}
  \label{macro:eq2}
    \left\{
    \begin{array}{l}
        \displaystyle \partial_t\cV
        = 
        N(\cV)
        -
        \cW
        -
       \mathcal{L}_{\rho_0}[\cV]\,
        ,\\[0.9em]
        \displaystyle \partial_t\cW
        \,=\, 
        A
        \left(\mathcal{V}\,,\,
        \mathcal{W}
        \right)
        \,,\\[0.9em]
        \displaystyle \partial_t \bar\nu
        \,-\,b\,
        \partial_{w}
        \left( 
        w \,\bar\nu\right) \,=\, 0\,,
    \end{array}
\right.
\end{equation*}
which turns out to be equivalent to \eqref{macro:eq} in the sense that it is solved by
$
\left(
\mathcal{V}\,,\,
\mathcal{W}\,,\,
\bar{\nu}
\right)
$
if and only if
$
\left(
\mathcal{V}\,,\,
\bar{\mu}
\right)
$
solves \eqref{macro:eq}, where $\bar{\mu}$ is defined as
\[
\bar{\mu}
\left(t\,,\,\bx\,,\,w
\right)
\,=\,
\bar{\nu}
\left(t\,,\,\bx\,,\,w\,-\,
\mathcal{W}(t\,,\,\bx)
\right)\,,
\]
for all 
$
\ds
(t,\,\bx,\,w) 
$ in $\R^+ \times K \times \R$.
Existence and uniqueness for $\bar{\nu}$ relies on classical arguments and we refer to \cite{limite-hydro-crevat}, where one can find the proof of existence and uniqueness for the system 
$
\ds
\left(\mathcal{V},\,
\mathcal{W}
\right)
$.
\end{proof}

\subsection{Main results}
\label{sec:22}
The following theorem is the main result of this article. It states that in the regime of strong interactions, the distribution of the voltage variable concentrates with rate $\sqrt{\eps}$ around $\mathcal{V}$ with Gaussian profile. The distribution of the adaptation variable converges towards $\bar{\mu}$. Hence, the couple 
$
\left(
\mathcal{V},\,
\bar{\mu}
\right)
$, which solves \eqref{macro:eq}, encodes the behavior of the system
when $\eps \ll 1$. {The result provides an explicit convergence rate
which is global in time and uniform in $\bx\in K$.}
\begin{theorem}\label{main:th}
Under assumptions \eqref{hyp1:N}-\eqref{hyp2:N} on the drift $N$, \eqref{hyp1:psi}-\eqref{hyp2:psi} on $\Psi$, \eqref{hyp:rho0}-\eqref{hyp2:f0} on the initial conditions $\mu^\eps_0$ and under the additional assumptions of Theorem \ref{WP mean field eq}, consider the solutions $\mu^\eps$ \& 
$
\left(
\mathcal{V},\,
\bar{\mu}
\right)
$ provided by Theorem \ref{WP mean field eq} \& \ref{wp macro eq} respectively. Furthermore, define the initial macroscopic and mesoscopic errors as
\[
\mathcal{E}_{\mathrm{mac}}
\,=\,
\left\|\,
\mathcal{U}_0
\,-\,
\mathcal{U}_0^\eps\,
\right\|_{L^{\infty}(K)}
\,+\,
\|\,
\rho_0
-
\rho_0^\eps\,
\|_{L^{\infty}(K)}\,,
\]
and
\[
\mathcal{E}_{\mathrm{mes}}
\,=\,
\sup_{\bx \in K}\,
W_2
\left(
\Bar{\nu}^\eps_{0,\bx},
\Bar{\nu}_{0,\bx}
\right)\,.
\]
Then there exists 
$(C,\eps_0)\, \in \, 
\left(
\R^+_*
\right)^2
$
such that the following expansion holds for all $\eps \leq \eps_0$,
\begin{align*}
\ds
W_2\left(
\mu^\eps_{t,\,\bx}\,,\, 
\mathcal{M}_{\frac{1}{\eps}\,\rho_0(\bx)\,,\,\mathcal{V}(t,\,\bx)}
\otimes
\bar{\mu}_{t,\,\bx}
\right)
\,
\leq\,
C
\,\left(
\min
\left(
e^{C\,t}
\left(\mathcal{E}_{\mathrm{mac}}
\,+\, \eps\right)\,,\,1
\right)
\,+\,
 \mathcal{E}_{\mathrm{mes}}\,e^{-b\,t}
\,+\,
e^{-\rho_0^\eps(\bx)\,t\,/\,\eps}
\right)
\,,
\end{align*}
for all $
(t,\bx) \in \R^+ \times K$.\\
Moreover, suppose the initial errors to be of order $\eps$, that is
\[
\ds
\mathcal{E}_{\mathrm{mac}}
\,+\,
 \mathcal{E}_{\mathrm{mes}}
\,
\underset{\eps \rightarrow 0}{=}
\,
O
\left(
\eps
\right)\,,
\]
then the following estimate holds, 
\begin{align*}
\ds
W_2\left(
\mu^\eps_{t,\,\bx}\,,\, 
\mathcal{M}_{\frac{1}{\eps}\,\rho_0(\bx)\,,\,\mathcal{V}(t,\,\bx)}
\otimes
\bar{\mu}_{t,\,\bx}
\right)
\,
\leq\,
C\,
\left(
\min
\left(
e^{C\,t}\,
\eps\,,\,1
\right)
\,+\,
e^{-\rho_0^\eps(\bx)\,t/\eps}
\right)
\,,\quad\forall
(t,\bx) \in \R^+ \times K\,.
\end{align*}
\end{theorem}
{An important feature in our work is that we \textbf{do not suppose} the initial condition $\mu^\varepsilon_{0}$ to be concentrated, nor the initial profile $\nu^\varepsilon_{0}$ to be close to its limit. In particular, this means that the problem on $\nu^\eps$ is \textbf{ill-prepared}. Indeed, performing the change of variable \eqref{change:var} in the following \textit{ansatz}
\[
\inf_{\bx \in K}
\int_{\R^2}
\,
|v
\,-\,
\mathcal{V}^\eps_0(\bx)|^2\,
\mu^\eps_{0,\bx}(\bu)\,\dD \bu
\,\geq\, 1\,,
\]
we obtain that $\nu^\eps_0$ blows up in $\scP_2$ as $\eps$ vanishes
\[
\inf_{\bx \in K}
\int_{\R^2}
\,
|v|^2\,
\nu^\eps_{0,\bx}(\bu)\,\dD \bu
\,\geq\, \eps^{-1}\,.
\]
This is why we prove in Proposition \ref{prop:2} some exponential
localizing effects with respect to the variable $t/\eps$, which also
appears in Theorem \ref{main:th}.} \\

Let us outline the main steps of the proof. First, we obtain in
Proposition \ref{prop:1} some uniform moment estimates for
$\mu^\eps$. Second, in Proposition \ref{prop:2}, we estimate a
relative energy, which should be interpreted as the moments of the
re-scaled quantity $\nu^\eps$ after a proper re-normalization, as
mentioned in Remark \ref{moment nu eps}. The last step consists in the
convergence estimate. We focus on the re-scaled quantity $\nu^\eps$
and develop an analytical coupling method in order to estimate its
Wasserstein distance with the limiting profile. The key idea is to
consider a coupled equation which is solved by couplings between
$\nu^\eps$ and its limit and then to estimate some energy for the
solutions to the coupled equation. {The improvement with respect to
\cite{limite-hydro-crevat} is twofold. On the one hand, we prove error
estimate not only on the macroscopic quantities but also on the distribution functions by using the Wasserstein
distance. On the other hand, by considering diffusive effect in $v$
and rescaled variables, we can characterize the asymptotic profile of
the distribution function and then get a better convergence rate.}

Let us now mention some interpretations and consequences of our result. The first consequence of the latter result is the convergence of the averaged quantities 
$
\mathcal{V}^\eps
$
and
$
\mathcal{W}^\eps
$.
In fact, we prove a finer result since we obtain that the couple 
$
\left(
\mathcal{V}^\eps,\,
\bar{\mu}^\eps
\right)
$
converges towards 
$
\left(
\mathcal{V},\,
\bar{\mu}
\right)
$
with rate $\eps$.
\begin{corollary}
Under the assumptions of Theorem \ref{main:th} and supposing the initial errors to be of order $\eps$, that is
\[
\ds
\mathcal{E}_{\mathrm{mac}}
\,+\,
 \mathcal{E}_{\mathrm{mes}}
\,
\underset{\eps \rightarrow 0}{=}
\,
O
\left(
\eps
\right)\,,
\]
there exists 
$(C,\eps_0)\, \in \, 
\left(
\R^+_*
\right)^2
$
such that for all $\eps \leq \eps_0$,
\begin{equation*}
\sup_{\bx \in K}
\left(
\left|\,
\mathcal{V}^\varepsilon
\left(t,\bx\right)
\,-\,\mathcal{V}
\left(t,\bx\right)\,
\right|
\,
+\,
W_2
\left(\bar{\mu}^\varepsilon_{t,\bx},\, 
\bar{\mu}_{t,\bx}
\right)
\right)
\,
\leq
\,
C\,
\left(
e^{- m_* \,t /\eps}
\,+\,
\min{\left(
e^{C\,t}
\eps,\, 1\right)}
\right)\,,
\end{equation*}
for all $t\,\geq\,0$, where $m_*$ is given by \eqref{hyp:rho0}.
\end{corollary}
\begin{proof}
Using Jensen's inequality,
we check that
\[
\left|
\mathcal{V}^\eps(t,\bx)
\,-\,
\mathcal{V}(t,\bx)
\right|
\,\leq\,
W_2\left(
\mu^\eps_{t,\bx},\,
\mathcal{M}_{\frac{\rho_0(\bx)}{\eps},\mathcal{V}(t,\bx)}
\otimes
\bar{\mu}_{t,\bx}
\right)\,.
\]
Furthermore, we use the definition of the Wasserstein distance to obtain
\[
W_2
\left(\bar{\mu}^\varepsilon_{t,\bx},\, 
\bar{\mu}_{t,\bx}
\right)
\,\leq\,
W_2\left(
\mu^\eps_{t,\bx},\,
\mathcal{M}_{\frac{\rho_0(\bx)}{\eps},\mathcal{V}(t,\bx)}
\otimes
\bar{\mu}_{t,\bx}
\right)\,.
\]
Then we apply Theorem \ref{main:th} and replace $\rho_0^\eps$ with its lower bound $m_*$ given in assumption \eqref{hyp:rho0}.
\end{proof}

Another interesting consequence of Theorem \ref{main:th}, which may be
interpreted as the expansion of $f^\eps$ at order $1$ when $\ds\eps
\ll 1$, consists in recovering order $0$. In fact, we prove a stronger
result. Indeed, let us make the analogy with other types of
expansions as Taylor expansions. The expansion at order $1$ yields an equivalence result at order $0$. This is exactly what we obtain in our case: we prove that the distance between $f^\eps$ and
$\ds
\delta_{\mathcal{V}}
\otimes
\bar{\mu}
$ is exactly of order $\ds\sqrt{\eps}$.
This justifies our approach for two reasons. First, it means that it
is not possible to achieve convergence at order $\eps$ if we restrict the analysis to the convergence towards a Dirac mass. Second, the choice of the Wasserstein metric in our analysis instead of another stronger norm enables to compare easily our result with the convergence of $f^\eps$ towards a Dirac mass.
\begin{corollary}\label{order 0}
Under the assumptions of Theorem \ref{main:th} and supposing the initial errors to be of order $\sqrt{\eps}$, that is
\[
\ds
\mathcal{E}_{\mathrm{mac}}
\,+\,
 \mathcal{E}_{\mathrm{mes}}
\,
\underset{\eps \rightarrow 0}{=}
\,
O
\left(
\sqrt{\eps}
\right)\,,
\]
there exists 
$(C,\eps_0)\, \in \, 
\left(
\R^+_*
\right)^2
$, such that for all $\eps \leq \eps_0$
\begin{equation*}
\sup_{\bx \in K}
\left(
W_2\left(
\mu^\eps_{t,\bx},\,
\delta_{\mathcal{V}(t,\bx)}
\otimes
\bar{\mu}_{t,\bx}
\right)
\right)
\,
\leq
\,
C\,
\left(
e^{- m_* \,t /\eps}
\,+\,
\min{\left(
e^{Ct}
\sqrt{\eps},\, 1\right)}
\right)\,,
\end{equation*}
for all $t\,\geq\,0$, where $m_*$ is given by \eqref{hyp:rho0}.\\
Moreover, supposing the initial errors to be of order $\eps$, that is
\[
\ds
\mathcal{E}_{\mathrm{mac}}
\,+\,
 \mathcal{E}_{\mathrm{mes}}
\,
\underset{\eps \rightarrow 0}{=}
\,
O
\left(
\eps
\right)\,,
\]
and considering two positive times $t_0$ and $T$ such that $t_0\,<\,T$, there exists 
$(C,\eps_0)\, \in \, 
\left(
\R^+_*
\right)^2
$
such that for all $\eps \leq \eps_0$
\begin{equation*}
C^{-1}
\sqrt{\eps}
\,\leq\,
W_2\left(
\mu^\eps_{t,\bx},\,
\delta_{\mathcal{V}(t,\bx)}
\otimes
\bar{\mu}_{t,\bx}
\right)
\,
\leq
\,
C
\sqrt{\eps}\,,
\quad
\forall
(t,\bx) \in [\,t_0\,,\,T\,] \times K\,.
\end{equation*}
\end{corollary}

\begin{proof}
The proof is a direct consequence of Theorem \ref{main:th} and the triangular inequality for $W_2$.
\end{proof}


\section{A priori estimates}
\label{sec:3}

In this section, we provide uniform estimates with respect to $\eps$ for the moments of $\mu^\eps$ and  for the relative energy given by
\begin{equation*}
\left\{
\begin{array}{l}
\ds M_{q}
\left[\,\mu^\eps\,
\right](t,\bx)
\,:=\,\int_{\R^2}
|\bu|^{q}
\,\dD\mu^\eps_{t,\bx}(\bu)\,,
\\[1.1em]
\ds D_{q}
\left[\,\mu^\eps\,
\right](t,\bx)\,:=\, 
\int_{\R^2} |v- \cV^\eps(t,\bx)|^{q} \,\dD\mu^\eps_{t,\bx}(\bu)\,,
\end{array}\right.
\end{equation*}
where $q\geq 2$.
\\

The key point here is to obtain uniform estimates with respect to time for both $M_{q}^\eps$ and $D_{q}^\eps$ using confining properties of $N$ and $A$. It is actually the only place where we use the super-linear {confinement} of the drift $N$  \eqref{hyp1:N}.

\begin{proposition}[Propagation of moment]
\label{prop:1}
Under assumptions \eqref{hyp1:N} on the drift $N$ and \eqref{hyp1:psi}-\eqref{hyp2:psi} on the interaction kernel $\Psi$, consider a sequence of solutions $(\mu^\eps)_{\eps\,>\,0}$ to \eqref{kinetic:eq} with initial conditions satisfying assumptions \eqref{hyp:rho0}-\eqref{hyp2:f0}. Then, for all positive $\eps$ and all $q$ lying in 
$
\ds
[\,2,\,2p\,]$ holds the following estimate
\[
M_{q}[\,\mu^\eps\,](t,\bx)
\,\,\leq\,\,
M_{q}[\,\mu^\eps\,](0,\bx)\,\exp\left(-{t}/{C}\right) \,+\, C, \quad\forall \,(t,\bx)\, \in\,\R^+\times K\,,
\]
where $C>0$ is a positive constant which only depends on $m_*$, $\ols{m}_p$ and the data of the problem: $\Psi$,~$A_0$ and $N$.
\end{proposition}
\begin{proof}
Let us choose an exponent $\theta \geq 2$, multiply equation \eqref{kinetic:eq} by $|\bu|^\theta/\theta$ and integrate with respect to $\bu \in \R^2$. Integrating by part, this leads us to the following relation
\begin{align*}
    \frac{1}{\theta}\,\frac{\dD}{\dD t}M_\theta[\mu^\eps](t,\bx)
\,    =\,
        \cI\,,
\end{align*}
where $\cI$ splits into
$
\cI \,=\,  \cI_{1}
\,+\, \cI_{2}
\,+\, \cI_{3}
\,+\, \cI_{4}
$ with 
\begin{equation*}
\left\{
\begin{array}{l}
    \displaystyle  \cI_{1} \,=\,
    -\frac{1}{\eps}\,
    \rho_0^\eps(\bx)
    \int_{\R^2}
    v\,|v|^{\theta-2}\,
    \left(v\,-\,\cV^\eps(t)\right)
    \,\mu^\eps_{t,\bx}(\bu)\,\dD \bu\,
    ,\\[0.9em]
    \displaystyle \cI_{2}
    \,=\,
    \int_{\R^2}v\,|v|^{\theta-2}\,N(v)
   \,\mu^\eps_{t,\bx}(\bu)\,\dD \bu\,,\\[0.9em]
    \displaystyle \cI_{3} \,=\,
    \int_{\R^2}
    \left(
    w\,|w|^{\theta-2}\, A(\bu)
    -v|v|^{\theta-2}\,w
    \,+\,
    (\theta-1)\,|v|^{\theta-2}
    \right)\,
    \mu^\eps_{t,\bx}(\bu)\,\dD \bu\,
    ,\\[0.9em]
    \displaystyle \cI_{4}\,=\,
    -
    \int_{\R^2} 
    v\,|v|^{\theta-2}\,
    \cK_{\Psi}[\rho_0^\eps\, \mu^\eps]
    \,\mu^\eps_{t,\bx}(\bu)\,\dD \bu\,.
    \end{array}
    \right.
    \end{equation*}
We first handle the stiff term $\cI_{1}$, which can be simply re-written as
\[
\cI_1
\,=\, -\frac{\rho_0^\eps(\bx)}{\eps}
\int_{\R^2}
\left(v\,|v|^{\theta-2}\,-\,\cV^\eps\,|\cV^\eps|^{\theta-2}\right)\,\left(v\,-\,\cV^\eps\right)\,\mu^\eps_{t,\bx}(\bu)\,\dD \bu
\,\leq\, 0\,.
\]
Then we evaluate $\cI_{2}$, which may involve higher order moments due to the non-linearity $N$. To overcome this difficulty, we split it in two parts
\[ 
v\,|v|^{\theta-2}\,N(v) \,=\, 
|v|^{\theta}\,\frac{N(v)}{v}\, 
\mathds{1}_{|v| \geq 1}
\,\,+\,\,
v\,|v|^{\theta-2}\,N(v)
\,\mathds{1}
_{|v| < 1}\,.
\]
According to assumption \eqref{hyp1:N} and since $N$ is continuous, we obtain
\[
\cI_{2}
\,\leq\,
\int_{\R^2}
\left(C\,-\,\omega^-(v)\right)\,
|v|^{\theta} \,\mu^\eps_{t,\bx}(\bu)\,\dD \bu
\,+\, C\,,
\]
for some constant $C>0$ only depending on $N$ and where $\omega^-$ is the following nonnegative function
\[
\omega^-(v)
\,=\,
\left(
\omega(v) \,\,\mathds{1}_{|v| \geq 1},
\right)^{-}\,,
\]
with $s^-=\max(0,-s)$ and where $\omega$ is given by \eqref{hyp1:N}.

Now we evaluate $\cI_{3}$, which gathers low order terms.
Applying Young's inequality, we obtain the following estimate
\[
\cI_{3}
\,\leq\,
\frac{C}{\eta^\theta}
\,\int_{\R^2}|v|^\theta \,\mu^\eps_{t,\bx}(\bu)\,\dD \bu
\,+\,
\int_{\R^2}(C\,\eta\,-\,b)\,|w|^\theta\,\mu^\eps_{t,\bx}(\bu)\,\dD \bu
\,+\,
\frac{C}{\eta^\theta}\,,
\]
for some constant $C>0$ and all $\eta \in ]0,1[$.

Finally, to evaluate the non-local term $\cI_{4}$, we estimate $\cV^\eps$ by applying Jensen's inequality, which yields
\[
\left|
\cV^\eps
\right|^\theta
\,\leq\,
\int_{\R^2}
|v|^\theta \,\mu^\eps_{t,\bx}(\bu)\,\dD \bu\,,
\]
hence, applying Young's inequality, we obtain that
\[
\cI_{4}
\,\leq\,
C\,
\left(
\|
\Psi
*_r
\rho_0^\eps
\|_{L^\infty(K)}
\,\int_{\R^2}
|v|^\theta\,\mu^\eps_{t,\bx}(\bu)\,\dD \bu
\,+\,
\int_{K\times \R^{2}}
\left|
\Psi(\bx,\bx')
\right|
\,|v|^\theta \,
 \rho_0^\eps(\bx')\, \mu^\eps_{t,\bx'}(\bu)\,\dD\bu\,\dD\bx'
\right)\,,
\]
where $C$ is a positive constant only depending on $\theta$. Then we use
Hölder's inequality and
assumption \eqref{hyp2:psi} \& \eqref{hyp:rho0} (we do not use the constraint $r>1$ here), which yields
\[
\|\,
\Psi
*_r
\rho_0^\eps\,
\|_{L^\infty(K)}\,
\leq\,
\|\,
\rho_0^\eps\,
\|_{L^\infty(K)}
\sup_{\bx \in K}
\|\,
\Psi(\bx,\cdot)\,
\|_{L^r(K)}\,
\leq\, C\,,
\]
for some $C>0$ independent of $\eps$.\\
In the former computations we choose $\eta$ such that $b\,-\,C\eta\,>\,0$. With the notation $\alpha\,=\,b\,-\,C\eta$, this yields
\begin{equation}
\label{Mp:0}
\begin{array}{ll}
\ds
\frac{1}{\theta}\,\frac{\dD}{\dD t}\,M_\theta[\,\mu^\eps\,](t,\bx)\, 
& 
\ds\leq\, 
\int_{\R^2}
\left[\,
\left(C
\,-\,
\omega^-(v)
\right)\,|v|^\theta
\,-\,\alpha\, |w|^{\theta}\,
\right]\, \mu^\eps_{t,\bx}(\bu)\,\dD \bu
\,+\,C
\\[0,8em]
&\ds\,+\,C\,\int_{K\times \R^{2}}
\left|
\Psi(\bx,\bx')
\right|
|v|^\theta
\,
 \rho_0^\eps(\bx')\, \mu^\eps_{t,\bx'}(\bu)\,\dD\bu\,\dD\bx'\,,
\end{array}
\end{equation}
for another constant $C>0$ depending on $\theta$, $m_*$, $A$, $N$ and $\Psi$ but not on 
$(t,\,\bx) \in \R^+ \times K$ nor on $\eps$. 

Now, we fix $q$ in $[2,\,2p]$ and proceed in two steps. On the one hand, choosing $\theta=qr'\geq 2$ in \eqref{Mp:0}, we evaluate the averaged moments $\ols{M}_{qr'}[\rho_0^\eps\, \mu^\eps]$ given by
\[
\ols{M}_{qr'}\left[\,
 \rho_0^\eps\, \mu^\eps\,\right](t)
\,=\,
\int_{K\times\R^2} |\bu|^{qr'}\,
 \rho_0^\eps(\bx)\, \mu^\eps_{t,\bx}(\bu)\,\dD\bu\,\dD\bx\,,
\]
where $r'$ is given by \eqref{hyp2:psi}.

On the other hand, choosing $\theta=q$ in \eqref{Mp:0}, we use the latter estimate to control the non-local contribution on the right hand side of \eqref{Mp:0} and evaluate $M_{q}[\mu^\eps](t,\bx)$ at each $(t,\bx)\in\R^+\times K$.

Starting from \eqref{Mp:0} with $\theta=qr'\geq 2$, we multiply it by $\rho_0^\eps(\bx)$ and integrate with respect to $\bx\in K$, which yields
\begin{align*}
\frac{1}{qr'}\,\frac{\dD}{\dD t}\ols{M}_{qr'}\left[\,
 \rho_0^\eps\, \mu^\eps\,\right]
&\,\leq\,
\int_{K\times \R^2}
\left(
\left(C
\,-\,
\omega^-(v)
\right)\,|v|^{qr'} 
-\,\alpha\,
|w|^{qr'}\right)\,
 \rho_0^\eps(\bx)\, \mu^\eps_{t,\bx}(\bu)\,\dD\bu\,\dD\bx
\\[0.9em]
&\,+\,
C\,\int_{K\times K\times \R^2}
|\Psi(\bx,\bx')|
\,|v|^{qr'}\,
\rho_0^\eps(\bx)
\,\,
 \rho_0^\eps(\bx')\, \mu^\eps_{t,\bx'}(\bu)\,\dD\bu\,\dD\bx'\,\dD\bx\,+\, C\,.
\end{align*}
According to assumption \eqref{hyp1:psi} \& \eqref{hyp:rho0}, we have 
\[
\int_{K\times K\times \R^2}
|\Psi(\bx,\bx')|
\,|v|^{qr'}\,
\rho_0^\eps(\bx)
\,\,
 \rho_0^\eps(\bx')\, \mu^\eps_{t,\bx'}(\bu)\,\dD\bu\,\dD\bx'\,\dD\bx
\,\leq \,C\,
\int_{K\times \R^2}
|v|^{qr'}\,
 \rho_0^\eps(\bx)\, \mu^\eps_{t,\bx}(\bu)\,\dD\bu\,\dD\bx\,,
\]
for some positive constant $C$ depending on $m_*$ and $\Psi$. Hence it yields
\[
\frac{1}{qr'}\,\frac{\dD}{\dD t}\,\ols{M}_{qr'}\left[\,
 \rho_0^\eps\, \mu^\eps\,\right](t)
\leq 
\int_{K\times \R^2}
\left(
\left(C
\,-\,
\omega^-(v)
\right)\,|v|^{qr'}
-\,\alpha\,
|w|^{qr'}\right)\,
 \rho_0^\eps(\bx)\, \mu^\eps_{t,\bx}(\bu)\,\dD\bu\,\dD\bx
\,+\,C\,.
\]
From assumption \eqref{hyp1:N},  $\omega^-(v)$ goes to infinity with $|v|$. Consequently, we are led to 
\[
\frac{1}{qr'}\,\frac{\dD}{\dD t}\, \ols{M}_{qr'}\left[\,
 \rho_0^\eps\, \mu^\eps\,\right](t)
\,\leq\,
C\,-\,\frac{1}{C}\, \ols{M}_{qr'}\left[\,
 \rho_0^\eps\, \mu^\eps\,\right](t),
\]
for $C>0$ great enough. Using Gronwall's lemma and the assumption \eqref{hyp2:f0} on the non-local moment of $\mu^\eps_0$, we obtain that for all $\eps > 0$,
\[
\ols{M}_{qr'}[\rho_0^\eps\, \mu^\eps](t) \,\leq\, C\,,\,\,
\forall t \in \R^+\,,
\]
where $C$ may depend on $\ols{m}_p$.

Now, we come back to the local estimate on the moment $M_{q}[\mu^\eps](t,\bx)$. We replace $\theta$ with $q$ in \eqref{Mp:0} and estimate the non-local contribution. First, we apply Hölder's inequality and assumptions \eqref{hyp2:psi} \& \eqref{hyp:rho0} to the non-local contribution in \eqref{Mp:0}. This gives
\[
\int_{K\times \R^{2}}
\left|
\Psi(\bx,\bx')
\right|
|v|^{q}\,
\rho_0^\eps(\bx')\, \mu^\eps_{t,\bx'}(\bu) \,\dD\bu\,\dD \bx'\,
\leq\,
C
\,
\left|
\ols{M}_{qr'}^\eps(t)
\right|^{1/r'}\,,
\]
where we emphasize that we use the constraint $r>1$ in the former estimate. 
Then we use the latter bound on $\ols{M}_{qr'}^\eps(t)$ and obtain
\[
\int_{K\times \R^{2}}
\left|
\Psi(\bx,\bx')
\right|
|v|^{q}\,
\rho_0^\eps(\bx')\, \mu^\eps_{t,\bx'}(\bu) \,\dD\bu\,\dD \bx'\,
\leq
C\,.
\]
Hence, it yields
\begin{equation*}
\frac{\dD}{\dD t}M_{q}[\mu^\eps](t,\bx) 
\,\leq\,
\int_{\R^2}
\left(
\left(C\,-\,\omega^-(v)
\right)\,|v|^{q}
\,-\,
\alpha
|w|^{q}
\right)\, \mu^\eps_{t,\bx}(\bu)\,\dD \bu \,+\,C\,.
\end{equation*}
Using the same arguments as before, we get some $C>0$ such that,
\[
\frac{\dD}{\dD t}
M_{q}[\mu^\eps](t,\bx) \,\leq\,
C\,-\,\frac{1}{C}\,M_{q}[\mu^\eps](t,\bx)\,,
\]
hence we conclude this proof applying Gronwall's lemma.
\end{proof}

As a straightforward consequence of Proposition \ref{prop:1}, we obtain uniform bounds with respect to time, space and $\eps$ for the macroscopic quantities
\begin{corollary}
\label{cor:1}
Under the assumptions of Proposition \ref{prop:1}, there exists a constant $C>0$, independent of $\eps$,  such that,
\[
\left|\cV^\eps(t,\,\bx)
\right| \,+\, \left|\cW^\eps(t,\,\bx)
\right|  \,\leq\, C, \quad \forall\,(t,\,\bx)\,\in\R^+\times K\,,
\]
where $\cV^\eps$ and $\cW^\eps$ are given by \eqref{macro:q}.
\end{corollary}
\begin{proof}
Applying the Cauchy-Schwarz inequality, we have for any $(t,\bx)\in\R^+\times K$,
\[
|\cV^\eps(t,\bx)| 
\,+\,
|\cW^\eps(t,\bx)|
\,\,\leq\,\, 2\,\left| M_2
[\mu^\eps]
(t,\bx)\right|^{1/2}\,,
\]
hence the result follows on from  Proposition \ref{prop:1}.
\end{proof}

We turn to the estimates for the relative energy $D_{q}[\mu^\eps]$, which quantifies the convergence of $\mu^\eps$ towards a Dirac mass centered on $\cV^\eps$.

\begin{proposition}[Relative energy]
\label{prop:2}
Under assumptions \eqref{hyp1:N}-\eqref{hyp2:N} on the drift $N$ and \eqref{hyp1:psi}-\eqref{hyp2:psi} on the interaction kernel $\Psi$, consider a sequence of solutions $(\mu^\eps)_{\eps\,>\,0}$ to \eqref{kinetic:eq} with initial conditions satisfying assumption \eqref{hyp:rho0}-\eqref{hyp2:f0}. There exists a positive constant $C>0$, which may depend on $m_p$ and $\ols{m}_p$, such that for all $\eps >0$ and all $q$ in 
$
\ds
[\,2,\,2p\,]$ holds the following estimate,
\[
D_q[\,\mu^\eps\,](t,\bx)
\,\,\leq\,\,
C \,\left[\,D_q[\,\mu^\eps\,](0,\bx) \,\exp \left(-\frac{q\rho_0^\eps(\bx)}{\eps}\,t \right)
\,+\,
\left(\frac{\eps}{\rho_0^\eps(\bx)}\right)^{q/2}\,\right], \quad\forall (t,\bx)\in\R^+\times K\,.
\]
\end{proposition}
\begin{proof}
We choose some $q$ in $[2,\,2p]$,
multiply equation \eqref{kinetic:eq} by  $|v-\cV^\eps|^q/q$ and integrate with respect to $\bu\in\R^2$. After integrating by part and using equation \eqref{macro:q}, it yields
\[
\displaystyle\frac{1}{q}\frac{\dD}{\dD t}
    D_q[\mu^\eps](t,\bx)
\,=\, \cJ\,,
\]
where $\cJ$ is split as
$$
\cJ\,=\,\cJ_{1}\,+\,\cJ_{2}\,+\,\cJ_{3}\,-\,
    \frac{\rho_0^\eps}{\eps}D_q[\mu^\eps](t,\bx)
    \,+\,
    (q-1)\,D_{q-2}^\eps[\mu^\eps](t,\bx)\,,
$$ 
with 
\begin{equation*}
\left\{
\begin{array}{l}
     \displaystyle  \cJ_{1} \,=\,
     -
   \,\int_{\R^2}
    (v\,-\,\cV^\eps)\,
    |v\,-\,\cV^\eps|^{q-2}\,
    \left(\cK_{\Psi}[\rho_0^\eps\,\mu^\eps]\,
    -\,
    \cL_{\rho_0^\eps}[\cV^\eps]
    \right)\,\mu^\eps_{t,\bx}(\bu)\,\dD \bu\,,
    \\[1.1em]
     \displaystyle \cJ_{2}
     \,=\,
    \int_{\R^2}
    (v\,-\, \cV^\eps)\,
    |v\,-\,\cV^\eps|^{q-2}\,
    \left(
    N(v)\,-\,N(\cV^\eps)
    \right)\,\mu^\eps_{t,\bx}(\bu)\,\dD \bu\,,\\[1.1em]
    \displaystyle \cJ_{3}\,=\,
    -\int_{\R^2}
    (v\,-\, \cV^\eps)\,
    |v\,-\,\cV^\eps|^{q-2}\,
    \left(\,
    w\,-\,\cW^\eps
    \,+\,
    \cE(\mu^\eps_{t,\bx})\,
    \right)\,\mu^\eps_{t,\bx}(\bu)\,\dD \bu\,.
    \end{array}
    \right.
    \end{equation*}
To estimate the non-local term $\cJ_{1}$, we observe that 
\[
\cK_{\Psi}[\rho_0^\eps\,\mu^\eps]
    \,-\, \cL_{\rho_0^\eps}(\cV^\eps)
\,=\,
(v\,-\,\cV^\eps)\,
\Psi *_r \rho_0^\eps(\bx)\,,
\]
hence, using assumptions \eqref{hyp2:psi} \& \eqref{hyp:rho0}, we obtain
\begin{align*}
    \cJ_{1}\,
    =\,-
    \,\Psi *_r \rho_0^\eps(\bx)\,
    \,D_q[\mu^\eps](t,\bx)
    \,\leq\, 
    C\, D_q[\mu^\eps](t,\bx)\,.
\end{align*}
We turn to $\cJ_{2}$, which  involves higher order moments since it displays the non-linearity $N$. On the one hand, Corollary \ref{cor:1} ensures that $\cV^\eps$ is uniformly bounded. On the other hand $N$ lies in $\scC^1(\R)$ and meets the confining assumption \eqref{hyp1:N}. Hence, there exists a constant $C>0$ independent of $\eps$ such that
\begin{equation*}
    (v-\cV^\eps)\,\left(N(v)\,-\,N(\cV^\eps)\right) 
    \,\leq\, C\,|v-\cV^\eps|^2\,, \quad \forall (t,\,\bx,\,v)\in\R^+\times K\times\R\,.
\end{equation*}
Thus, it yields
\[
\cJ_{2}
    \,\leq\,
    C\, D_{q}[\mu^\eps](t,\bx)\,.
\]
Finally we estimate $\cJ_{3}$, which gathers the low order terms. According to assumption \eqref{hyp2:N} on $N$  and applying Proposition \ref{prop:1}, we have for all positive $\eps$,
\[
\int_{\R^2}
    \left|
    w\,-\,\cW^\eps
    \right|^q
    \mu^\eps_{t,\bx}(\bu)\,\dD \bu
           \,+\,
|\cE(\mu^\eps_{t,\bx})|
\,\leq\, C\,, \quad \forall\,(t,\,\bx) \in \R^+ \times K\,,
\]
for some constant $C$ that may depend on $m_p$ and $\ols{m}_p$.
Hence, applying H\"older's inequality, it yields
\[
\cJ_{3}
\,\leq\, C \, D_q[\mu^\eps]^{{(q-1)}/{q}}(t,\bx)
\,\leq\,
C\, \left( D_q[\mu^\eps](t,\bx) \,+\,  D_{q}[\mu^\eps]^{{(q-2)}/{q}}(t,\bx) \right)\,.
\]
Gathering these computations and applying H\"older's inequality to $D_{q-2}^\eps$, we obtain
\[
\frac{1}{q}\,\frac{\dD}{\dD t}D_q[\mu^\eps](t,\bx)
\,+\, \frac{\rho_0^\eps}{\eps}D_q[\mu^\eps](t,\bx)
\,\leq\,
C
\,\left(
D_{q}[\mu^\eps](t,\bx)
\,+\, D_{q}[\mu^\eps]^{{(q-2)}/{q}}(t,\bx)
\right)\,.
\]
To estimate $D_{q}[\mu^\eps]$, we introduce the function $
\ds u\,=\,
\left(
D_q[\mu^\eps]
\right)^{2/q}$, which satisfies the following differential inequality
\[
\frac{1}{2}\,\frac{\dD u}{\dD t}
\,+\,\frac{\rho_0^\eps}{\eps}\,u
\,\leq\, C\,\left( u \,+\, 1\right)\,.
\]
Applying Proposition \ref{prop:1}, we get a first bound on $u$ since
$$
D_q[\mu^\eps](t,\bx)
\,\leq\, C\, \,M_{q}[\mu^\eps](t,\bx)
\,\leq\, C\,,
$$
hence, we substitute this estimate on the right hand side of the former differential inequality and obtain
\[\frac{\dD u }{\dD t}
\,+\,
\frac{2\,\rho_0^\eps}{\eps}\,u
\,\leq\, C\,,
\]
which implies that
\[
u(t) \,\leq\, u(0)\,\exp
\left(-\frac{2\rho_0^\eps\,t}{\eps}
\right)
\,+\, 
C\,\frac{\eps}{\rho_0^\eps}\, \left( 1 \,-\,\exp\left(-\frac{2\,\rho_0^\eps\,t}{\eps}\right)\right)\,.
\]
The result follows from replacing $u$ by $\left(D_q[\mu^\eps]\right)^{2/q}$.
\end{proof}

Both Propositions \ref{prop:1} \& \ref{prop:2} may be interpreted in terms of the re-scaling $\nu^\eps$ according to the following remark

\begin{remark}\label{moment nu eps}
Performing the change of variable \eqref{change:var} in the expression of $M_q[\mu^\eps]$ and $D_q[\mu^\eps]$, we obtain the following relations
\[
\left\{
\begin{array}{l}
\ds\int_{\R^2} |v|^q \, \nu^\eps_{t,\bx}(\bu)\,\dD\bu \,=\, \eps^{-q/2} D_q[\mu^\eps](t,\bx)\,,
\\[1.1em]
\ds\int_{\R^2}|w|^q \, \nu^\eps_{t,\bx}(\bu)\,\dD\bu \,=\, \int_{\R^2}|w-\cW^\eps|^q \, \mu^\eps_{t,\bx}(\bu)\,\dD\bu\,\leq\, C\,M_q[\mu^\eps](t,\bx)\,.
\end{array}\right.
\]
\end{remark}

To conclude this section, we deduce from Propositions \ref{prop:1} \& \ref{prop:2} the following error estimate

\begin{proposition}\label{estimate for the error 2}
Under assumptions \eqref{hyp1:N}-\eqref{hyp2:N} on the drift $N$, \eqref{hyp1:psi}-\eqref{hyp2:psi} on the interaction kernel $\Psi$ consider a sequence of solutions $(\mu^\eps)_{\eps\,>\,0}$ to \eqref{kinetic:eq} with initial conditions satisfying assumption \eqref{hyp:rho0}-\eqref{hyp2:f0}. There exists a constant $C>0$ such that for all $\eps >0$ we have
\[
\left|\mathcal{E}(\mu^\eps_{t,\bx})\right| 
\leq 
C\left(
e^{-2\rho_0^\eps(\bx) t/\eps}
+
\eps\right)\,,\quad\forall
(t,\bx) \in \R^+ \times\,K\,.
\]
\end{proposition}

\begin{proof}
We rewrite the error term $\mathcal{E}(\mu^\eps_{t,\bx})$ given by \eqref{error} as follows
\[
\mathcal{E}( \mu^\eps_{t,\bx})
\,=\,
\mathcal{E}_1
\,+\,
\mathcal{E}_2
\,+\,
\mathcal{E}_3\,,
\]
where
\[
\left\{
\begin{array}{l}
\ds \mathcal{E}_1 \,=\, 
N'\left(
\mathcal{V}^\eps
\right) \int_{\R^2} 
\left(
v
-
\mathcal{V}^\eps
\right)
\mathds{1}_{\,|v
-
\mathcal{V}^\eps|\, \leq\, 1}
\,\mu^\eps_{t,\bx}(\bu)\,
\dD \bu\,,
\\[1.1em]
\ds \mathcal{E}_2 \,=\, \int_{\R^2} 
\left(
N\left(
v
\right)
-
N\left(
\mathcal{V}^\eps
\right)
-
N'\left(
\mathcal{V}^\eps
\right)
\left(
v
-
\mathcal{V}^\eps
\right)
\right)
\mathds{1}_{\,|v
-
\mathcal{V}^\eps| \,\leq\, 1}
\,\mu^\eps_{t,\bx}(\bu)\,
\dD \bu\,,
\\[1.1em]
\ds \mathcal{E}_3 \,=\, \int_{\R^2} 
\left(
N\left(
v
\right)
-
N\left(
\mathcal{V}^\eps
\right)
\right)
\mathds{1}_{\,|v
-
\mathcal{V}^\eps|\, > \,1}
\,\mu^\eps_{t,\bx}(\bu)\,
\dD \bu\,.
\end{array}\right.
\]
We start with $\mathcal{E}_1$. According to the definition of $\mathcal{V}^\eps$, we have
\[
\mathcal{E}_1
=
0\,
-\,
N'\left(
\mathcal{V}^\eps
\right)
\int_{\R^2} 
\left(
v
-
\mathcal{V}^\eps
\right)
\mathds{1}_{\,|v
-
\mathcal{V}^\eps|\, > \, 1}
\,\mu^\eps_{t,\bx}(\bu)\,
\dD \bu\,.
\]
Since $N \in \scC^1(\R)$ and applying Corollary \ref{cor:1} \& Proposition \ref{prop:2}, it yields
\[
\left|
\mathcal{E}_1
\right|
\,\leq\,
C\,
\left(
e^{-2\rho_0^\eps t/\eps}
\,+\,
\eps
\right)\,.
\]
We turn to $\mathcal{E}_2$. Since $N$ is of class $\scC^2$, and using Corollary \ref{cor:1}, we have
\[
\left|
N\left(
v
\right)
-
N\left(
\mathcal{V}^\eps
\right)
-
N'\left(
\mathcal{V}^\eps
\right)
\left(
v
-
\mathcal{V}^\eps
\right)
\right|
\mathds{1}_{\,|v
-
\mathcal{V}^\eps| \,\leq\, 1}
\,\leq\,
C\,|v
-
\mathcal{V}^\eps|^2\,,
\]
for some constant $C>0$ independent of $\ds(t,\bx,v) \, \in \, \R^+ \times K \times \R$ and $\eps >0$. Hence, we obtain the following estimate applying Proposition \ref{prop:2}
\[
\left|
\mathcal{E}_2
\right|
\,\leq\,
C\,
\left(
e^{-2\rho_0^\eps t/\eps}
\,+\,
\eps
\right)\,.
\]
To end with, we estimate $\mathcal{E}_3$ using the assumption \eqref{hyp2:N} on $N$ and Corollary \ref{cor:1}. Indeed, we have
\[
\left|
N\left(
v
\right)
-
N\left(
\mathcal{V}^\eps
\right)
\right|
\mathds{1}_{\,|v
-
\mathcal{V}^\eps|\, > \,1}
\,\leq\,
C\,|v
-
\mathcal{V}^\eps|^p\,.
\]
Hence, applying Proposition \ref{prop:2}, we obtain
\[
\left|
\mathcal{E}_3
\right|
\leq
C \,\left( \,e^{-p\rho_0^\eps t/\eps}
\,+\,
\eps^{p/2}\right)
\]
and gathering these computations, it yields  the expected result.
\end{proof}

In the present section, we derived pointwise estimates for the moments of $\mu^\eps$, a relative energy which corresponds to the moments of $\nu^\eps$ and the macroscopic error term $\mathcal{E}(\mu^\eps)$. We build on these results to prove Theorem \ref{main:th}.


\section{Proof of Theorem \ref{main:th}}
\label{sec:4}

To achieve the proof of Theorem \ref{main:th}, we quantify the concentration of  $(\mu^\eps)_{\eps>0}$ around its asymptotic profile $\mathcal{M}_{\rho_0/\eps,
\mathcal{V}}
\otimes
\bar{\mu}$ when $\eps\rightarrow 0$. Therefore, this section consists in estimating the error term  $\ds W_2\left(\mu^\eps,
\mathcal{M}_{\rho_0/\eps,
\mathcal{V}}
\otimes
\bar{\mu}
\right)$,
which we decompose as follows: for any $(t,\bx)\in\R^+\times K$,
\[
W_2^2
\left(
\mu^\eps_{t,\, \bx},\,\mueq\right)\,
\leq\, \cD_1\,+\, \cD_2 \,+\, \cD_3,
\]
where $ \cD_1$, $ \cD_2$ and $ \cD_3$ are defined as
$$
\left\{
\begin{array}{l}
\ds\cD_1 \,:=\, 
W_2^2 \left( \mu^\eps_{t,\,\bx}\,,\,
\mathcal{M}_{\rho_0^\eps/\eps
,\,\mathcal{V}^\eps}
\otimes \bar{\nu}_{t,\,\bx}
\left(
\cdot \,-\, \mathcal{W}^\eps
\right)
\right)\,,
\\[1.1em]
\ds\cD_2
\,:=\, 
W_2^2
\left(
\mathcal{M}_{\rho_0^\eps/\eps,\mathcal{V}^\eps}
\otimes 
\bar{\nu}_{t,\bx}
\left(
\cdot - \mathcal{W}^\eps
\right)\, ,\,
\mathcal{M}_{\rho_0^\eps/\eps,\mathcal{V}}
\otimes 
\bar{\mu}_{t,\bx}
\right)\,,
\\[1.1em]
\cD_3\,:=\, 
W_2^2
\left(
\mathcal{M}_{\rho_0^\eps/\eps,\mathcal{V}}
\otimes 
\bar{\mu}_{t,\bx}\,,\,
\mathcal{M}_{\rho_0/\eps,\mathcal{V}}
\otimes 
\bar{\mu}_{t,\bx}
\right)\,.
\end{array}\right.
$$
The term $\cD_1$ quantifies the distance between the concentration profile  $\nu^\eps$ and the Gaussian distribution since from the change of variable \eqref{change:var}, we have
$$
\cD_1 \,=\, \inf_{\pi^\eps\in\Pi^\eps}\int_{\R^4} 
\left(\eps \left|v-v' \right|^2 
+
\left|w-w' \right|^2 \right)
 \dD \pi^\eps (\bu,\bu')\,,
 $$
where $\Pi^\eps$ stands for the set of distributions over $\R^4$ with marginals $\nu^\eps_{t,\bx}$ and $\cM_{\rho^\eps_0}\otimes \bar{\nu}_{t,\bx}$,   whereas $\cD_2$  quantifies the error between the macroscopic quantities. Indeed, with the same change of variable, we obtain
\[
\cD_2
=
\left|\mathcal{V}^\eps
-
\mathcal{V}
\right|^2
+
\left|
\mathcal{W}^\eps
-
\mathcal{W}
\right|^2.
\]
Finally, the term $\cD_3$ is simply the error on the macroscopic density. In the following we give an estimate for each term $(\cD_i)_{1\leq i\leq 3}$.

\subsection*{Estimate for $\cD_1$}

Our strategy to estimate $\ds\cD_1$ consists in introducing a coupled equation on $\pi^\eps \in \Pi^\eps$. Let us consider the following problem, 
\begin{equation}
\label{combined problem}
    \begin{array}{ll}
 &   \displaystyle\partial_t \pi^\eps 
    \,+\,
    \text{div}_{\bu}
    \left[
    \mathbf{b}^\eps_0(t,\bx,\bu)\,
    \pi^\eps
    \right]
    \,+\,
\partial_{w'}
\left[A_0\left(0,w' \right)\pi^\eps \right]
\\[1.1em]
&    \,=\,
    \ds
    \frac{1}{\eps}\,\text{div}_{v,v'}
    \left[
    \rho_0^\eps(\bx)
    \left(
    v,\,v'
    \right)^\top
\pi^\eps
    \,+\,
    D \cdot
    \nabla_{v,v'} \pi^\eps
    \right],
    \end{array}
\end{equation}
where $\mathbf{b}^\eps_0$ is given in \eqref{nu:eq} and the diffusion matrix $D$ is given by
\[
D
\,=\,
\begin{pmatrix}
1 & \beta\\
\beta & 1 \\
\end{pmatrix}\,,
\]
for some $\beta \in [-1,1]$. On the one hand, integrating the former equation with respect to $\bu'=(v',w')$, we obtain equation \eqref{nu:eq} on $\nu^\eps$. On the other hand, equation \eqref{combined problem} integrated with respect to $\bu$ is given by
\begin{equation*}
    \displaystyle\partial_t \,\nu
    \,+\,\partial_{w'}
\left[A_0\left(0,w' \right)\nu \right]
    \,=\,
    \frac{1}{\eps}\partial_{v'}
    \left(
    \rho_0^\eps v'\, \nu
    +
    \partial_{v'}\, \nu
    \right)\,,
\end{equation*}
which is solved by $\mathcal{M}_{\rho_0^\eps}\otimes \bar{\nu}$.
Consequently, if we take some initial data $\pi^\eps_0$ such that
\[
\pi^\eps_0
\in \Pi
\left(
\nu^\eps_{0,\bx}\,,\, 
\mathcal{M}_{\rho_0^\eps}\otimes \bar{\nu}_{0,\bx}
\right)\,,
\]
we obtain that the solution  $\pi^\eps$ to \eqref{combined problem} has marginals $\nu^\eps_{t,\bx}$ with respect to $\bu$ and $\mathcal{M}_{\rho_0^\eps} \otimes \bar{\nu}_{t,\bx}$ with respect to $\bu'$ at all time $t \geq 0$. Hence, according to the definition of the Wasserstein metric, the following inequality holds
\[
\cD_1
\,\leq\,
\int_{\R^4} 
\left(\eps \left|v-v' \right|^2 
\,+\,
\left|w-w' \right|^2 \right)
 \dD \pi^\eps (\bu,\bu')\,.
\]
Moreover, we say that the equation is coupled because of the diffusion matrix $D$.
Condition $\beta~\in~[-1,1]$ ensures that the matrix is positive and hence is required in order for equation \eqref{combined problem} to be well-posed. In the case $\beta=0$, there is no coupling between variables $\bu$ and $\bu'$. However, we can see that there is only one suitable choice for the parameter $\beta$. Indeed, considering only the leading order in equation \eqref{combined problem}, the equation rewrites
\begin{equation*}
    \begin{array}{ll}
    \ds
    \displaystyle\partial_t \pi^\eps 
    \,=\,
    \ds
    \frac{1}{\eps}\,\text{div}_{v,v'}
    \left[
    \rho_0^\eps 
    \left(
    v,\,v'
    \right)^\top
\pi^\eps
    \,+\,
    D \cdot
    \nabla_{v,v'} \pi^\eps
    \right]\,.
    \end{array}
\end{equation*}
Multiplying the former equation by $|v-v'|^2/\,2$ and integrating by part, we obtain
\[
\frac{1}{2}\,\frac{\dD}{\dD t}\int_{\R^4} 
\left|v-v' \right|^2 
 \dD \pi^\eps (\bu,\bu')
 \,+\,
 \frac{\rho_0^\eps}{\eps}
 \int_{\R^4} 
\left|v-v' \right|^2 
 \dD \pi^\eps (\bu,\bu')
\,=\,
2\,\frac{1-\beta}{\eps}\,.
\]
Consequently, $\beta=1$ is the unique suitable choice in order to avoid blow up in the asymptotic
$\eps \, \rightarrow\, 0$. Hence, \textbf{we fix parameter} $\beta$ to $1$ in what follows.
On a probabilistic point of view, this choice corresponds to taking the same Brownian motion for the processes associated to $\nu^\eps$ and $\mathcal{M}_{\rho_0^\eps} \otimes \bar{\nu}$. \\
Before estimating $\cD_1$, we precise the nature of solutions we consider for equation \eqref{combined problem}
\begin{definition}\label{notion de solution combined problem}
For any $\bx$ in $K$ and some $\varepsilon>0$, we say that $\pi^\varepsilon$ solves \eqref{combined problem} with initial condition $\pi^\varepsilon_0$ if we have
\begin{enumerate}
\item  $\pi^\eps$ lies in
\[ 
\scC^0 \left(
\R^+\,,\,
\cD'
\left(\R^4\right)
\right)
\cap 
L^\infty_{loc}
\left(
\R^+\,,\,
L\,\log{
L
}
\left(\R^4\right)
\right)\,
,\]
where $\cD'(\R^4)$ stands for the set of distributions over $\R^4$ and $L\,\log{
L
}$ stands for the set of $L^1$ functions with finite entropy over $\R^4$.
\item for all $t \geq 0$,
and $\ds\varphi \in \scC_c^\infty
\left(\R^4\right)$, it holds
\begin{align*}
\int_{\R^4}
\varphi\,
\left(
\pi^\varepsilon_{t}
-
\pi^\varepsilon_{0}
\right)\,\dD\bu\,
\dD\bu'\,
&=\,
\int_0^t
\int_{\R^4}
\left(
\nabla_{\bu} \varphi
\cdot 
\mathbf{b}^\eps_0(s,\bx,\bu)
        \,+\,
\partial_{w'}
 \varphi\,A_0(0,\, w') 
\right)
\pi^\varepsilon_{s}\,\dD\bu\,\dD\bu'\,\dD s\\[0,8em]
&-\frac{1}{\eps}
\int_0^t
\int_{\R^4}
\left(
\rho_0^\eps(\bx)
\nabla_{v,v'}
\varphi \cdot
\left(
v,v'\right)^\top
\,-\,
\mathrm{div}_{(v,v')}
\left(
D\cdot
\nabla_{v,v'}
\varphi
\right)
        \right)
\pi^\varepsilon_{s}\,\dD\bu\,\dD\bu'\,\dD s\,,
\end{align*}
\end{enumerate}
where $\mathcal{V}^\eps$ and $\mathbf{b}^\eps_0$  are given by \eqref{macro:q} and \eqref{kinetic:eq}.
\end{definition}
We prove the following result for equation \eqref{combined problem}
\begin{proposition}\label{WP combined problem}
Under the assumptions of Theorem \ref{WP mean field eq}, consider a positive $\eps$, some $\bx$ lying in $K$ and a coupling 
\[
\pi^\eps_0~
\in~\Pi
\left(
\nu^\eps_{0,\bx}\,,\, 
\mathcal{M}_{\rho_0^\eps}\otimes \bar{\nu}_{0,\bx}
\right).\]
There exists a solution $\pi^\eps$ to equation \eqref{combined problem} with initial condition $\pi^\eps_0$ and parameter $\beta=1$ in the sense of Definition \ref{notion de solution combined problem}. Furthermore, we have 
\[
\pi^\eps(t,\cdot)
\in \Pi
\left(
\nu^\eps_{t,\bx}\,,\,
\mathcal{M}_{\rho_0^\eps}\otimes \bar{\nu}_{t,\bx}
\right),\quad
\forall t \in \R^+\,.
\]
\end{proposition}
We postpone the proof to Appendix \ref{proof wp combined pb}. It is mainly technical since equation \eqref{combined problem} is linear.\\

We come back to the main concern of this section which consists in estimating $\cD_1$ and prove the following estimate
\begin{proposition}
\label{estimate D 1}
Under the assumptions of Theorem \ref{main:th}, there exist two positive constants $
\ds C
$ and $
\ds\eps_0
$ such that for all $\eps \leq \eps_0$, holds the following estimate
\begin{align*}
\cD_1
\,\leq\,
C\,
\left(\,
W_2^2
\left(\,
\Bar{\nu}^\eps_{0,\bx}\,,\,
\Bar{\nu}_{0,\bx}\,
\right) e^{-2\,b\, t}\,
+\,
e^{-2\,\rho_0^\eps \,t /\eps}
\,+\,
\eps^2\,
\right)\,,
\quad
\forall
(t,\bx)\in   \R^+ \times K\,
.
\end{align*}
\end{proposition}
\begin{proof} We consider 
$
\pi^\eps_0
\in \Pi
\left(
\nu^\eps_{0,\bx}\,,\,
\mathcal{M}_{\rho_0^\eps}\otimes \bar{\nu}_{0,\bx}
\right)
$
and the associated solution $\pi^\eps$ to equation \eqref{combined problem} given by Proposition \ref{WP combined problem}. 
When the context is clear, we omit the dependence with respect to $t$.
To evaluate the term $\cD_1$, we introduce the following quantities
$$
\left\{
\begin{array}{l}
\displaystyle
\ds\cA(t) \,:=\, \int_{\R^4} \left|v-v' \right|^2\,\dD \pi^\eps (\bu,\bu')\,,
\\[1.1em]
\ds\cB(t) \,:=\, \int_{\R^4}\left|w-w' \right|^2 \,\dD \pi^\eps (\bu,\bu')\,.
\end{array}\right.
$$
Observing that $\cD_1(t) \leq \eps\,\cA(t) \,+\, \cB(t)$, this proof consists in showing that $\cA(t)$ is of order $\eps$ whereas $\cB(t)$ is of order $\eps^2$.
\\ 
We begin with $\cA$ and multiply equation \eqref{combined problem} by 
$\ds
|v-v'|^2/2
$
and integrate by part with respect to $(\bu,\bu')$, it yields
\begin{align*}
    \frac{1}{2}\frac{\dD}{\dD t}\displaystyle 
    \cA
    \,+\,
    \frac{\rho_0^\eps}{\eps}\cA
    \displaystyle=
    \cA_{1}\,,
\end{align*}
where
\begin{equation*}
     \displaystyle  \cA_{1} \,=\,
     \frac{1}{\sqrt{\eps}}
   \,\int_{\R^4}
    (v-v')\,
    \left(
    N(\mathcal{V}^\eps
    +
    \sqrt{\eps}\, v)
    \,-\,
    w
    \,-\,
    \sqrt{\eps}\,v\, \Psi*_r\rho_0^\eps(\bx)
    \right)\,
    \dD\pi^\eps (\bu,\bu')\,.
    \end{equation*}
We apply Cauchy-Schwarz inequality, assumptions \eqref{hyp2:N}, \eqref{hyp2:psi} \& \eqref{hyp:rho0} and Corollary \ref{cor:1} to estimate $\cA_1$, it yields
\[
\cA_1
\,\leq\,
\frac{C}{\sqrt{\eps}}
\,\cA(t)^{1/2}
\,\left(
\int_{\R^2}
    \left(
    1\,
    +\,
    |w|^2
    \,+\,
    |\sqrt{\eps}\,v|^{2p}
    \right)\,
    \dD\nu^\eps_{t,\bx} (\bu)
\right)^{1/2}\,.
\]
Then we perform the change of variable \eqref{change:var} in the latter integral and apply Propositions \ref{prop:1} \& \ref{prop:2} and obtain
\begin{align*}
    \ds\frac{1}{2}\frac{\dD}{\dD t}
    \cA(t)
    \,+\,
    \frac{\rho_0^\eps}{\eps}\,\cA(t)
    \displaystyle
    \,\leq\,
    \frac{C}{\sqrt{\eps}}\,\cA(t)^{1/2}\,.
\end{align*}
Dividing the former estimate by $\ds \cA(t)^{1/2}$ and applying Gronwall's lemma, this yields
\[
\cA(t)^{1/2}
\,\leq\,
\cA(0)^{1/2}\,
\exp\left(
-\rho_0^\eps\, t/\eps
\right)
\,+\,
C\,\eps^{1/2}\,,
\]
where the constant $C>0$ may depend on the lower bound $m_*$ of the spatial distribution $\rho_0^\eps$ (see assumption \eqref{hyp:rho0}), $m_p$ and $\bar{m}_p$. We point out that since we do not prepare the initial condition, $\cA(0)$ may blow up as $\eps$ vanishes. Indeed, we have
\[
\cA(0)
\,\leq\,
2\,
\left(
\int_{\R^2}
    |v|^{2}
    \,
    \dD\nu^\eps_{0,\bx} (\bu)
\,+\,
\frac{1}{\rho^\eps_0(\bx)}
\right)\,.
\]
Hence, operating the change of variable \eqref{change:var} in the latter integral and applying assumptions \eqref{hyp:rho0} \& \eqref{hyp1:f0}, the former estimate becomes
\[
\cA(0)
\,\leq\,
\frac{C}{\eps}\,.
\]
Therefore, for all $\eps$ less than $1$, we obtain
\[
\cA(t)
\leq
\frac{C}{\eps}
\left(\exp{\left(
-2\,\rho_0^\eps(\bx) \,t/\eps
\right)}
\,+\,\eps^2
\right)\,.
\]
We turn to $\cB$ and prove that it is of order $\eps^2$ using the previous estimate on $\cA$. Indeed, we compute the derivative of $\cB$ multiplying equation \eqref{combined problem} by $|w-w'|^2/2$ and integrating by part with respect to $(\bu,\bu')$, it yields
\begin{align*}
\frac{1}{2}\,\frac{\dD}{\dD t} 
\cB(t)
&\,=\,
-\,b\, \cB(t)
\,+\,
a\,\sqrt{\eps}
\,\left(
\cB_1(t)
\,+\,
\cB_2(t)
\right)\,,
\end{align*}
where
\begin{equation*}
\left\{
\begin{array}{l}
     \displaystyle  \cB_{1}(t) \,=\,
     \int_{\R^4}
(v-v')\,(w-w')\,\dD \pi^\eps(\bu,\bu')\,,
    \\[1.1em]
    \displaystyle \cB_2(t)\,=\,
    \int_{\R^4}
v'\,
(w-w')\,  
\dD \pi^\eps(\bu,\bu')\,.
    \end{array}
    \right.
    \end{equation*}
In order to estimate $\cB_1$, we apply Cauchy-Schwarz inequality, use the estimate on $\cA(t)$ and assumption \eqref{hyp:rho0} on $\rho_0^\eps$, which yields
\[
a\,\sqrt{\eps}\,
\cB_1
\,\leq\,
C\,
\left(
\exp{
\left(
-m_*\, t/\eps
\right)}\,
+\,
\eps
\right)
\,\cB(t)^{1/2}\,,
\]
for some $C>0$ and where  $m_*$ is given in \eqref{hyp:rho0}. 
The next step consists in proving that $\cB_2(t)$ is of order $\eps^{3/2}$. We compute the derivative of $\cB_2(t)$ multiplying equation \eqref{combined problem} by $\ds
v' (w-w')\,  $ and integrating by part with respect to $(\bu,\bu')$, it yields
\begin{equation*}
    \frac{\dD}{\dD t}
    \cB_2(t)
\,=\,
-\left(\frac{\rho_0^\eps}{\eps}
\,+\,
b \right)
\cB_2(t)
\,+\,
a\,\sqrt{\eps}\,
\int_{\R^4}
v\,v'\,
\dD \pi^\eps
(\bu,\bu').
\end{equation*}
Then we apply Young's inequality, invert the change of variable \eqref{change:var} and apply Proposition \ref{prop:2}. It yields
\begin{equation*}
    a\,\sqrt{\eps}\,
\int_{\R^4}
\left|
v\,v'\right|\,
\dD \pi^\eps
\left(
\bu,\, \bu'
\right)
\,\leq\,
C\,\left(
\eps^{-1/2}\,e^{-2\rho_0^\eps t/\eps}
\,+\,
\eps^{1/2}
\right),
\end{equation*}
where the positive constant $C$ may depend on $m_*$, $m_p$ and $\bar{m}_p$.
Then, we multiply the equation on $\ds\cB_2$ by its sign and apply Gronwall's lemma. In the end, this leads to
\[
\left|\cB_2(t)\right|
\leq
\left|\cB_2(0)\right|\exp\left(
-\rho_0^\eps t/\eps
\right)
+
C
\left(
\eps^{1/2}
\exp\left(
-\rho_0^\eps t/\eps
\right)
+
\eps^{3/2}
\right).
\]
for any $\eps$ less than $m_*/(2b)$.
Furthermore, applying Cauchy-Schwarz inequality in $\cB_2(0)$ and using assumption \eqref{hyp:rho0}, we obtain
\[
\left|
\cB_2(0)
\right|\,
\leq\,
C\,\cB(0)^{1/2}.
\]
Gathering the former computations and applying assumption \eqref{hyp:rho0}, we obtain 
\begin{equation*}
\frac{1}{2}\frac{\dD}{\dD t}\, \cB(t) 
+
b\, \cB(t)
\,\leq\,
C\,
\left(
e^{
-m_* t/\eps}\,
+\,
\eps\,
\right)
\cB(t)^{1/2}\,
+\,
C
\left(
\eps^{1/2}\,
\cB(0)^{1/2}\,
e^{
-m_* t/\eps}\,
+\,
\eps\,
e^{
-m_* t/\eps}\,
+\,
\eps^2\,
\right)\,.
\end{equation*}
To estimate $\cB(t)$, we will construct an upper-bound $\cB_+(t)$ by considering the following ODE
\begin{equation*}
\left\{
\begin{array}{ll}
\ds
\frac{1}{2}\,\frac{\dD}{\dD t}\, \cB_+(t) 
\,+\,
b \,\cB_+(t)\,
=
\,
2\,C
\,\left(
e^{
\left(b
\,-\,
m_*/\eps
\right)\, t}\,
+\,
\eps
\right)
\cB_+(t)^{1/2}\,,
\\[1,1em]
\ds
\cB_+(0)^{1/2}\,=\,
\cB(0)^{1/2}\,+\, \frac{2\,C\,\eps}{b}\,,
\end{array}
\right.
\end{equation*}
whose exact solution is given by
\[
\cB_+(t)^{1/2}
\,=\,
\cB_+(0)^{1/2}\,e^{-bt}
\,+\,
\frac{2\,C\,\eps}{b}
\left(
1\,-\,e^{-bt}
\right)
\,+\,
\frac{2\,C\,\eps}{m_*\,-\,2b\eps}
\,\left(
e^{-bt}
\,-\,
e^{-
\left(
m_*/\eps
\,-\, b 
\right)\,t}\,
\right)\,.
\]
We check that the following condition is fulfilled at all time $t$, 
\begin{equation*}
\ds
\cB_+(t)^{1/2}\,\geq\,\eps^{1/2}\,\cB(0)^{1/2}\,e^{-b\,t}\,+\,\eps\,,
\end{equation*}
as long as $\ds 2\,C\,\geq\,b$ and $\ds \eps \,<\, \min
{\left\{\,m_*\,/\,(2b)\,,\,1\,\right\}}$.
Making use of the latter inequality, we get
\[
\eps^{1/2}\,
\cB(0)^{1/2}
\,e^{
-m_* t/\eps}\,
+\,
\eps
\,e^{
-m_* t/\eps}\,
+\,
\eps^2\,
\leq\,
\left(
e^{(b\,-\,m_*/\eps) t}
\,+\,
\eps
\right)
\,\cB_+(t)^{1/2}\,.
\]  
Therefore, defining $u(t) =\cB_+(t)\,-\,\cB(t)$, we check that the following inequality holds
\[
\frac{1}{2}\frac{\dD u}{\dD t} 
\,+\,
b\, 
u\,
\geq\,
C\,
\frac{e^{
-m_* t/\eps}\,
+\,
\eps}{
\cB(t)^{1/2}\,+\,\cB_+(t)^{1/2}
}\,
u\,.
\]
Since $\cB$ is non negative and $\ds\cB_+(t)$ stays lower bounded by $\eps^2$, we apply Gronwall's lemma to the latter estimate and noticing that $u(0)\,\geq\,0$, it yields that $u(t)$ is non-negative. Then, we deduce
\[ \cB(t) 
\,\leq \, 
\cB_+(t)
\,\leq\, 
2\,\cB(0)\,e^{-2bt}
\,+\,
C\,\eps^2\,,
\]
for all $t \geq 0$, 
as long as $\ds 2\,C\,\geq\,b$ and $\ds \eps \,<\,\min
{\left\{\,m_*\,/\,(2b)\,,\,1\,\right\}}$.\\
Gathering the former computations, we obtain the following estimate for $\cD_1$
\[
\cD_1
\,\leq\,
C\,
\left(
\cB(0)\, e^{-2b t}
\,+\,
e^{
-2 \,\rho_0^\eps\, t/\eps
}
\,+\,
\eps^2
\right)\,.
\]
We conclude the proof taking the infimum over all $\pi^\eps_0$.
\end{proof}

\subsection*{Estimate for $\cD_2$}
We mentioned that the change of variables \eqref{change:var} in $\cD_2$ yields
\[
\cD_2
\,=\,
|\mathcal{V}-\mathcal{V}^\eps|^2
\,+\,
|
\mathcal{W}-
\mathcal{W}^\eps
|^2\,.
\]
Hence the proof consists in injecting the error estimate obtained in Proposition \ref{estimate for the error 2} in equations \eqref{macro-eps:eq} \& \eqref{macro:eq}
\begin{proposition}
\label{estimate D 2}
Under the assumptions of Theorem \ref{main:th}, there exist two positive constants $
\ds C
$ and $
\ds\eps_0
$ such that for all $\eps \leq \eps_0$, holds the following estimate
\[
\cD_2^{1/2}
\,\leq\,
C\,
\min
{
\left(
e^{C\,t}
\left(
\mathcal{E}_{\mathrm{mac}}\,
+\,
\eps
\right),
1
\right)},
\quad
\forall\,(t,\bx) \in \R^+ \times K\,,
\]
where 
$\mathcal{E}_{\mathrm{mac}}$ is defined in Theorem \ref{main:th}.
\end{proposition}
\begin{proof}
We omit the dependence with respect to $(t,\bx)$ when the context is clear and write $\ds\|\cdot\|_{\infty}$ instead of $\ds\|\cdot\|_{L^\infty(K)}$ in this proof.
According to equations \eqref{macro-eps:eq} \& \eqref{macro:eq}, we have
\begin{equation*}
\left\{
\begin{array}{ll}
\ds
\frac{\dD}{\dD t}
\left|
\mathcal{V}
-
\mathcal{V}^\eps
\right|&=\ds\,
\sign(\mathcal{V}
-\mathcal{V}^\eps)
\left[
N(\mathcal{V})-N(\mathcal{V}^\eps)
-
\left(
\mathcal{W}-\mathcal{W}^\eps
+
\mathcal{L}_{\rho_0}
\left(
\mathcal{V}\right)
-
\mathcal{L}_{\rho_0^\eps}
\left(
\mathcal{V}^\eps
\right)
\right)
-
\mathcal{E}(\mu^\eps)
\right]\,,\\[0.9em]
\ds
\frac{\dD}{\dD t}
\left|\mathcal{W}
-\mathcal{W}^\eps
\right|
&=\ds
\sign(\mathcal{W}
-\mathcal{W}^\eps)\,
A_0
\left(\,\mathcal{V}\,-\,\mathcal{V}^\eps\,,\,
\mathcal{W}\,-\,\mathcal{W}^\eps\,\right)\,,
\end{array}
\right.
\end{equation*}
where $\sign(v) = v\,/\,|\,v\,|$ for all $v \in \R^*$. According  to Corollary \ref{cor:1} and Theorem \ref{wp macro eq}, $
\ds
\left(\mathcal{V},\mathcal{V}^\eps\right)$ is uniformly bounded by a constant $R>0$ with respect to $\eps$, time \& space. Consequently, we have
\[
(\mathcal{V}\,-\,\mathcal{V}^\eps)
\left(
N(\mathcal{V})\,-\,N(\mathcal{V}^\eps)
\right)\,
\leq\, C\,|\mathcal{V}\,-\,\mathcal{V}^\eps|^2\,,
\]
where $C$ stands for the Lipschitz constant of $N$ over the ball of radius $R$.
We now estimate the contribution of the non-local terms. Using the linearity of $\mathcal{L}$ we split the term as follows
\[
\mathcal{L}_{\rho_0}
\left(
\mathcal{V}
\right)\,
-\,
\mathcal{L}_{\rho_0^\eps}
\left(
\mathcal{V}^\eps
\right)\,
=\,
\mathcal{L}_{\rho_0}
\left(
\mathcal{V}\,
-\,
\mathcal{V}^\eps
\right)\,
+\,
\mathcal{L}_{(\rho_0-\rho_0^\eps)}
\left(
\mathcal{V}^\eps
\right)\,.
\]
According to assumption \eqref{hyp2:psi} (
we do not use the constraint $r>1$ here), and since $\mathcal{V}^\eps$ 
is uniformly bounded (see Corollary \ref{cor:1}), we have
\[
-\sign
\left(
\mathcal{V}\,
-\,
\mathcal{V}^\eps
\right)\,
\mathcal{L}_{(\rho_0-\rho_0^\eps)}
\left(
\mathcal{V}^\eps
\right)\,
\leq\,
C
\,
\left\|
\rho_0\,-\,\rho_0^\eps
\right\|_{\infty}\,.
\]
Furthermore, according to assumptions \eqref{hyp2:psi} \& \eqref{hyp:rho0} (we do not use the constraint $r>1$ here),  we obtain 
\[
-\sign
\left(
\mathcal{V}\,
-\,
\mathcal{V}^\eps
\right)\,
\mathcal{L}_{\rho_0}
\left(
\mathcal{V}
\,
-\,
\mathcal{V}^\eps
\right)\,
\leq\,
C\,
\left|\mathcal{V}-\mathcal{V}^\eps
\right|\,
+\,
\left|
\Psi
\right|
*_r
\left(
\left|\mathcal{V}-\mathcal{V}^\eps
\right|
\rho_0
\right)\,.
\]
We estimate the non-local term using assumptions \eqref{hyp2:psi} \& \eqref{hyp:rho0} (we do not use the constraint $r>1$ here). It yields
\[
\left|
\Psi
\right|
*_r
\left(
\left|\mathcal{V}-\mathcal{V}^\eps
\right|
\rho_0
\right)
\,\leq\,
C\,
\left\|\mathcal{V}-\mathcal{V}^\eps
\right\|_{\infty}\,,
\]
for some constant $C$ only depending on $m_*$ and $\Psi$.
Then we gather the former computations and replace $\mathcal{E}(\mu^\eps)$ by the bound obtained in Proposition \ref{estimate for the error 2}. It yields
\begin{equation*}
\frac{\dD}{\dD t}
\left|\,
\mathcal{U}
\,-\,
\mathcal{U}^\eps\,
\right|
\,
\leq\,
C
\left(
\left\|\,
\mathcal{U}
\,-\,
\mathcal{U}^\eps\,
\right\|_{\infty}
\,+\,
\eps\,
+\,
e^{-2\,m_*\,t/\eps}\,
+\,
\left\|\rho_0\,-\,\rho_0^\eps
\right\|_{\infty}
\right)\,,
\end{equation*}
for some positive constant $C$ which may depend on $m_p$, $\ols{m}_p$ and $m_*$ but not on $\eps$ and $(t,\bx)$. Integrating the latter inequality between $0$ and $t$ and taking the supremum over all $\bx$ in $K$, we end up with the following inequality
\begin{equation*}
\left\|\,
\mathcal{U}(t)
\,-\,
\mathcal{U}^\eps(t)\,
\right\|_{\infty}
\, \leq\,
\left\|\,
\mathcal{U}_0
\,-\,
\mathcal{U}^\eps_0\,
\right\|_{\infty}
\,+\,
C
\int_0^t
\left\| \mathcal{U}(s)
\,-\,
\mathcal{U}^\eps(s)
\right\|_{\infty}
\,+\,
\eps\,
+\,
e^{-m_*\, s/\eps}\,
+\,
\left\|\rho_0\,-\,\rho_0^\eps
\right\|_{\infty}
\,\dD s\,.
\end{equation*}
We conclude the proof applying Gronwall's lemma and using that $\mathcal{U}$ and $\mathcal{U}^\eps$ are uniformly bounded according to Theorem \ref{wp macro eq} and Corollary \ref{cor:1}.
\end{proof}

\subsection*{Estimate for $\cD_3$}
We now turn to the last term $\cD_3$. This section is only technical and we prove that $\cD_3$ is negligible in comparison to $\cD_1$ and $\cD_2$.
{ Indeed, since $\cD_3$ is the distance between two tensors, it is bounded by the sum of the distances appearing in each tensor
\[
\cD_3\,\leq\, 
W_2^2
\left(
\mathcal{M}_{\rho_0^\eps/\eps,\mathcal{V}}
\,,\,
\mathcal{M}_{\rho_0/\eps,\mathcal{V}}
\right)
\,+\,
W_2^2
\left(
\bar{\mu}_{t,\bx}\,,\,
\bar{\mu}_{t,\bx}
\right)
\,=\,
W_2^2
\left(
\mathcal{M}_{\rho_0^\eps/\eps,\mathcal{V}}
\,,\,
\mathcal{M}_{\rho_0/\eps,\mathcal{V}}
\right)
\,.
\]}
Then, changing variables, we obtain
\[
\cD_3\,\leq\, 
\frac{2\,\eps}{\rho_0^\eps}\, W_2^2\left(\mathcal{M}_1,\mathcal{M}_1\right)
\,+\, 2\,\eps\,
\left(\frac{1}{\sqrt{\rho_0^\eps}}
\,-\,\frac{1}{\sqrt{\rho_0}}
\right)^2\,.
\]
Hence, according to assumption \eqref{hyp:rho0}, we obtain
\[
\cD_3
\,
\leq \,
\frac{\eps}{2\,m_*^{3}}
\,\|
\rho_0\,
-
\,
\rho_0^\eps
\|_{L^{\infty}(K)}^2\,.
\]

Therefore applying Propositions \ref{estimate D 1} and \ref{estimate D 2} together with the latter  estimate on $\cD_3$, we have proven  Theorem \ref{main:th}.

\section{Conclusion \& Perspectives}

In this paper, we  have characterized the blow-up profile (Gaussian distribution) of the voltage distribution in the regime of strong \& short-range coupling between neurons and have  computed the limiting distribution for the adaptation variable as well. Our result should be interpreted as the first order expansion  of the network's distribution as $\eps$ vanishes.  Indeed, it allows to improve on former convergence result in the sense that we gain an order $\sqrt{\eps}$ in our convergence rate. On top of that, we benefit from the first order expansion  to derive an asymptotic equivalent of the distribution at order zero (see Corollary \ref{order 0}).\\

Let us also mention a few natural questions that arise from this work. The natural continuation of this article consists in obtaining a strong convergence result towards the concentration profile (see \cite{blaustein:prep}). It appears that the relative entropy approach we developed in Appendix \ref{Appendix} may adapt to the analysis of the asymptotic $\eps \ll 1$. We also point out that it should be possible to derive the next corrective terms in the expansion of the macroscopic quantities thanks to this work. Indeed, based on our result with a slight improvement, it may be possible to prove
\[
\ds
\left(\,
\mathcal{V}^\eps,\,
\mathcal{W}^\eps\,
\right)
\,
\underset{\epsilon
\rightarrow
0
}{=}
\,
\left(\,
\ols{
\mathcal{V}}^\eps,\,
\ols{
\mathcal{W}}^\eps\,
\right)
\,+\,
O
\left(
\eps^{3/2}
\right),
\]
where the limiting macroscopic system $\left(\,
\ols{
\mathcal{V}}^\eps,\,
\ols{
\mathcal{W}}^\eps\,
\right)$ solves the following system
\begin{equation*}
    \left\{
    \begin{array}{l}
        \displaystyle \partial_t\,
        \ols{
\mathcal{V}}^\eps
        \,=\, 
        N\left(
        \ols{
\mathcal{V}}^\eps
        \right)
        \,-\,
        \ols{
\mathcal{W}}^\eps
        \,-\,
       \mathcal{L}_{\rho_0}[\,\ols{
\mathcal{V}}^\eps\,]
        \,+\,
        \frac{\eps}{2}\,
        \rho_0\,
        N''\left(
        \ols{
\mathcal{V}}^\eps
        \right)
        ,\\[0.9em]
        \displaystyle \partial_t\,
        \ols{
\mathcal{W}}^\eps
        \,=\, 
        A
        \left(\ols{
\mathcal{V}}^\eps\,,\,
        \ols{
\mathcal{W}}^\eps
        \right)
        .
    \end{array}
\right.
\end{equation*}
The corrective term adds a dependence with respect to the spatial distribution of neurons and it might add some complexity to the dynamics of the limiting macroscopic system. {Our last comment on this work is that it might be possible to improve it approach in order to obtain uniform in time convergence estimate. This could be done by carrying stability analysis of the limiting macroscopic system \eqref{macro:eq}. Indeed, going back to the proof of Theorem \ref{main:th}, the only estimate which is not uniform with respect time is the one given for $\cD_2$ (see Proposition \ref{estimate D 2}), which corresponds to the error between the macroscopic quantities $\cU$ and $\cU^\eps$. Therefore, it should be possible to obtain some uniform in time convergence results by taking $\cU^\eps$ close to an equilibrium state of the limiting macroscopic equation $\eqref{macro:eq}$.
}


\section*{Acknowledgment}
Francis Filbet gratefully acknowledges the support of  ANITI (Artificial and Natural Intelligence Toulouse Institute). 
This project has received support from ANR ChaMaNe No: ANR-19-CE40-0024.


\appendix
\addcontentsline{toc}{section}{Appendices}

\section{Proof of Theorem \ref{WP mean field eq}}\label{Appendix}
In this section, we prove existence and uniqueness of solution for equation \eqref{kinetic:eq} in the sense of Definition \ref{notion de solution}. The main difficulty is to prove uniqueness and continuity of the solution
$$
\mu^\eps \in \scC^0\left(\R^+\times K\,,\,L^1\left(\R^2\right)\right)\,. 
$$
To this aim, we  first establish {\it a priori} estimates on the solution by propagating  exponential moments in $\bfu\in\R^2$. Then the key point of the proof is to define a modified relative entropy $H_\alpha\left[\,\mu\, |\,\nu\,\right]$, given for any $\alpha$ in $]0,1[$ as
    \begin{equation}
      \label{def:Halpha}
H_{\alpha}\left[\,\mu\, |\,\nu\,\right]\,=\,\int_{\R^2}
\mu
\,\ln{
\left(
\frac{
\mu
}{
\alpha\mu + (1-\alpha)\nu
}
\right)
}
\,\dD\bu\,, 
\end{equation}
for two non-negative functions
$
\mu,\,
\nu
\,\in\,
L^1
\left(
\R^2
\right)
$  with mass equal to one.
Relative entropy has several technical advantages in comparison to the $L^1$ norm. On the one hand, it makes explicit the dissipation due to the Laplace operator, which we refer as the Fisher information in what follows. On the other hand, the latter modified relative entropy, where the denominator is a convex combination of $\mu$ \& $\nu$ is well defined for all positive functions $\mu$ and $\nu$ with mass $1$, with no further condition on the domain on which they vanish. Indeed, under these conditions, holds the following inequality
\[
{0 \,\leq}\, H_{\alpha}\left[\,
\mu\, |\,
\nu\,
\right]
\,\leq\,
-\ln{(\alpha)}\,<\,+\infty.
\]
We even prove the following stronger result.
\begin{lemma}
  \label{entropy V.S. L1}
For any two non-negative functions
$\mu$,  $\nu \in L^1(\R^2)$ with integral equal to one, the following estimate holds
\begin{equation}
\frac{(1\,-\, \alpha)^2}{2}
\left\|
\mu\,-\,
\nu
\right\|_{L^1
\left(
\R^2
\right)}^2
\,\leq\,
H_{\alpha}
\left[
\mu\,|\,
\nu
\right]
\,\leq\,
\frac{1-\alpha}{\alpha}
\left\|
\mu
\,-\,
\nu
\right\|_{L^1(\R^2)}\,.
\end{equation}
\end{lemma}
\begin{proof}
The first inequality is a straightforward consequence of the
Csiz\'ar-Kullback inequality: for $\mu$,  $\nu \in L^1(\R^2)$ with integral equal to one
$$
\| \mu - \nu \|_{L^1}^2 \,\leq \, 2 \, H_0[\mu\,|\,\nu].
$$
{The result is obtained applying the latter inequality with $\nu= \kappa^\alpha$, where $\kappa^\alpha$ is given by
\[
\kappa^\alpha
\,=\,
\alpha \mu 
\,+\,
(1\,-\,
\alpha
)\nu\,.
\]}

For the second inequality,{\tiny } we notice that 
$
\ds
\left|\,\mu\,
/
\kappa^\alpha\,
\right|\,\leq\,
\alpha^{-1}
$.
Then we apply the convex inequality $\ln{(1\,+\,x)}\,\leq\,x$ to the following relation
\[
H_\alpha
\left[\,
\mu\,|\,
\nu
\,\right]
\,=\,
\int_{\R^2}
\,
\mu
\ln{
\left(
1\,+\,
\frac{\mu
\,-\,
\kappa^\alpha
}{
\kappa^\alpha
}
\right)
}
\,
\dD \bu
\,.
\]
\end{proof}
From this modified relative entropy functional, we  prove the continuity of the solution with respect to both the time and the spatial variable in the functional space $L^1$. Now, we fix $\alpha=1/2$ and denote by $I_{1/2}$ the associated relative Fisher information
\begin{equation}
  \label{def:Ialpha}
I_{1/2}\left[\mu\,|\,\nu\right] \,:=\, \int_{\R^2}\left|\partial_v\ln{\left(\frac{2\mu}{\nu+\mu}\right)}\right|^2\mu\,\dD\bu\,.
\end{equation}
To deal with existence issues, we provide an entropy estimate
$$
H[\mu^\eps_{t,\bx}] \,:=\, \int_{\R^2}\mu^\eps_{t,\bx}(\bu)\ln(\mu^\eps_{t,\bx}(\bu)) \,\dD\bu\,,
$$
and follow a classical regularization argument.
\subsection{A priori estimates}
\label{estimates wp}
In this section, we suppose that we have a smooth solution $\mu^\eps_{t,\bx}$ to equation \eqref{kinetic:eq} and provide exponential moment, entropy and continuity estimates for $\mu^\eps_{t,\bx}$.

We first define  $J[\mu^\eps_{t,\bx}]$ as
\begin{equation*}
 J[\mu^\eps_{t,\bx}]
\,:=\, 
\,\int_{\R^2}
e^{|\bu|^2/2}
\,\mu^\eps_{t,\bx}(\bu)\,\dD \bu\,,
\end{equation*}
and prove the following result.

\begin{proposition}[Exponential moments]\label{exp moments}
For any $\varepsilon > 0$, suppose that assumptions \eqref{hyp1:N} and \eqref{hyp1:psi}-\eqref{hyp:rho0} are fulfilled whereas the initial condition $\mu_0^\varepsilon$  verifies the first condition of \eqref{hyp3:f0}. Then, for  any solution $\mu^\eps$ to equation \eqref{kinetic:eq}, there exists a positive constant $C$ that may depend on $\eps$ such that
\[
J[\mu^\eps_{t,\bx}]
\,\,\leq\,\,
C\,, \quad\forall \,(t,\bx)\, \in\,\R^+\times K\,.
\]
\end{proposition}

{\begin{remark}
	We emphasize that if we consider a simplified model with homogeneous adaptation variable $w$, it is possible to prove that exponential moment are not only propagated as in Proposition \ref{exp moments} but also created, thanks to the super-linear confining properties  \eqref{hyp1:N}-\eqref{hyp2:N} of the drift $N$. However, this approach does not work when considering the full model. Indeed, since the dynamics are linear with respect to $w$, one can not expect any gain with respect to $w$. In addition, due to the crossed terms between the $v$ and the $w$ variables displayed in equation $\eqref{kinetic:eq}$ on $\mu^\eps$, it seems hard to estimate exponential moments with respect to the $v$ and $w$ variables separetely.\\
	On top of that, one can see in the proof of Proposition \ref{exp moments} that super-linear (or at least linear with large enough coefficient) confining properties are required on $N$ to control the term $\cJ_2$ (see below) and particularly the crossed terms and the convolution term that it displays.
\end{remark}}

\begin{proof}
We multiply equation \eqref{kinetic:eq} by $e^{|\bu|^2/2}$ and integrate with respect to $\bu \in \R^2$, after an integration by part, it yields
\begin{align*}
    \,\frac{\dD}{\dD t}J[\mu^\eps_{t,\bx}]
\,    =\,
        \cJ\,,
\end{align*}
where $\cJ$ is split into
$
\cJ \,=\,  \cJ_{1}
\,+\, \cJ_{2}
\,+\, \cJ_{3}\,,
$ with 
\begin{equation*}
\left\{
\begin{array}{l}
    \displaystyle  \cJ_{1} \,=\,
    \left(
    \frac{1}{\eps}\,\left(
    \rho_0^\eps\,
    \cV^\eps
    \right)\,
    +\,
    \Psi*_r
    \left(
    \rho_0^\eps\,
    \cV^\eps
    \right)
    \right)(t,\bx)\,
    \int_{\R^2}
    v
    \,
    e^{|\bu|^2/2}\,
    \mu^\eps_{t,\bx}(\bu)\,\dD \bu
   \, ,\\[1em]
    \displaystyle \cJ_{2} \,=\,
    \int_{\R^2}
    \left(
    w\, A(\bu)
    \,-\,v\,w
    \,-\,
    v^2\,
    \left(
    \frac{\rho_0^\eps(\bx)}{\eps}
    +
    \Psi*_r \rho_0^\eps(\bx)
    -
    1
    \right)
    +
    1
    \right)
    e^{|\bu|^2/2}\,
    \mu^\eps_{t,\bx}(\bu)\,\dD \bu
    \,,\\[1em]
    \displaystyle \cJ_{3} \,=\,
    \int_{\R^2}
    v \, N(v)
    \,
    e^{|\bu|^2/2}\,
    \mu^\eps_{t,\bx}(\bu)\,\dD \bu
    \,.
    \end{array}
    \right.
\end{equation*}
The  proof follows the same lines as the one given in Proposition \ref{prop:1} on the moment estimates except that here the non-local term $\cJ_1$ can be roughly estimated. Indeed, applying Corollary \ref{cor:1} and assumptions \eqref{hyp2:psi} \& \eqref{hyp:rho0}, we first obtain
\[
\cJ_{1}
\,\leq\,
C
\,\int_{\R^2}
|v|\,
e^{|\bu|^2/2}\,\mu^\eps_{t,\bx}(\bu)\,\dD \bu\,,
\]
for some positive constant 
$\ds C\,>\,0$
that may depend on $\eps$.
Then we estimate $\cJ_2$ and $\cJ_3$ following the computations in the proof of Proposition \ref{prop:1} and taking advantage of the confinement property \eqref{hyp1:N} on $N$. In the end, we get
\begin{equation*}
\frac{\dD}{\dD t}J[\mu^\eps_{t,\bx}] 
\,\leq\,
\int_{\R^2}
\left(
\left(C\,-\,\omega^-(v)
\right)|v|^2
\,-\,
\alpha
|w|^2
+
C
\right)
e^{|\bu|^2/2}
\, \mu^\eps_{t,\bx}(\bu)\,\dD \bu\,,
\end{equation*}
where $C$ and $\alpha$ are two positive constants that may depend on $\eps$ and where $\omega^-$ is defined as
\[
\omega^-(v)
\,=\,
\left(
\omega(v) \,\,\mathds{1}_{|v| \geq 1}
\right)^{-}\,,
\]
where $\omega$ is given in \eqref{hyp1:N}. Finally, according to \eqref{hyp1:N}, we have 
\[
\lim_{|\bu| \rightarrow + \infty}
\left(C\,-\,\omega^-(v)
\right)|v|^2
\,-\,
\alpha
|w|^2
+
C
\,=\,
-\infty\,,
\]
hence it gives 
\[
\,\frac{\dD}{\dD t} J[\mu^\eps_{t,\bx}]
\,\leq\,
C\,-\,\frac{1}{C}\, J[\mu^\eps_{t,\bx}]\,.
\]
We conclude the proof applying Gronwall's lemma and using the first assumption in \eqref{hyp3:f0}.
\end{proof}

Then we provide an entropy estimate of the solution $\mu^\eps_{t,\bx}$.

\begin{proposition}[Entropy estimates]
For any $\varepsilon > 0$,  suppose that assumptions \eqref{hyp1:N}-\eqref{hyp2:N} and \eqref{hyp1:psi}-\eqref{hyp:rho0} are fulfilled whereas the initial condition $\mu_0^\varepsilon$  verifies the second condition in \eqref{hyp3:f0}. Then, for  any solution $\mu^\eps$ to equation \eqref{kinetic:eq}, there exists a positive constant $C>0$ that may depend on $\eps$ such that
\[
H[\mu^\eps_{t,\bx}]
\,\,\leq\,\,
H[\mu^\eps_{0,\bx}]\,
+
\, C\,t\,, \quad\forall \,(t,\bx)\, \in\,\R^+\times K\,.
\]
\end{proposition}
{\begin{remark}
	The latter result raises the natural question of wether the entropy estimate holds uniformly with respect to time. An answer was given in  \cite[Theorem 2.2]{mlimit3} in a simplified setting: authors consider a spatially homogeneous network and treat the case where $N$ is a cubic function. They provide uniform estimate for the norm of the solution to \eqref{kinetic:eq} in regular norms, namely in $H^1$ and $H^1\cap H_v^2$ spaces. It might be possible to adapt their argument to a low regularity setting such as ours and to obtain some uniform in time estimates on the entropy. However, this is an interesting question on its own and since the present work focuses on the asymptotic $\eps \rightarrow 0$ rather than on the long time behavior, we did not follow this path here.
\end{remark}}
\begin{proof}
We multiply equation \eqref{kinetic:eq} by $\ln{
\left(
\mu^\eps_{t,\bx}
\right)
}$ and integrate with respect to $\bu \in \R^2$. Then we use conservation of mass for \eqref{kinetic:eq} and integrate by part. It yields
\begin{align*}
    \,\frac{\dD}{\dD t}H[\mu^\eps_{t,\bx}]
    \,+\,
I\left[
\mu^\eps_{t,\,\bx}
\right]
\,    =\,
\frac{\rho_0^\eps}{\eps}
\,
+\,
\Psi*_r\rho_0^\eps\,
+\,
b
\,+\,
\int_{\R^2}\,
N(v)\,
\partial_v
\left(
\ln{
\mu^\eps_{t,\bx}
}
\right)\,
\mu^\eps_{t,\bx}\,\dD \bu\,,
\end{align*}
where the Fisher information $I$ is given by
\[
I\left[
\mu^\eps_{t,\,\bx}
\right]
\,=\,
\int_{\R^2}
\left|
\partial_v\,
\ln{
\left(
\mu^\eps_{t,\,\bx}
\right)
}
\right|^2\,
\mu^\eps_{t,\,\bx}\,
\dD \bu\,.
\]
According to Young's inequality, we have
\[
\int_{\R^2}\,
N(v)\,
\partial_v
\left(
\ln{
\mu^\eps_{t,\bx}
}
\right)\,
\mu^\eps_{t,\bx}\,\dD \bu
\,\leq\,
\frac{C}{\eta}
\int_{\R^2}\,
\left|N(v)\right|^2\,
\mu^\eps_{t,\bx}\,\dD \bu
\,+\,
\eta\,
I
\left[
\mu^\eps_{t,\bx}
\right]\,,
\]
for all positive $\eta$. We choose $\eta\,=\,1/2$, apply Proposition \ref{prop:1}, and assumptions \eqref{hyp2:N}, \eqref{hyp2:psi} \&  \eqref{hyp:rho0}. This yields
\begin{align*}
    \,\frac{\dD}{\dD t}H[\mu^\eps_{t,\bx}]
\,    \leq\,
C\,,
\end{align*}
for some $C>0$ that may depend on $\eps$. Then we integrate the former inequality with respect to time and obtain the result.
\end{proof}

We now turn to continuity estimates with respect to the spatial variable and close the estimates in $L^1$ making use of the modified relative entropy $H_\alpha\left[\,
\mu\, |\,
\nu\,
\right]$ given in \eqref{def:Halpha}.

\begin{proposition}[Continuity with respect to $\bx$]
  \label{continuity:space}
For any $\varepsilon > 0$,  suppose that assumption \eqref{hyp1:N} and \eqref{hyp1:psi}-\eqref{hyp:rho0} are fulfilled whereas the initial condition $\mu_0^\varepsilon$  verifies the first condition of \eqref{hyp3:f0}. Then, for  any solution $\mu^\eps$ to equation \eqref{kinetic:eq}, the following estimate holds
\[
\left\|
\mu^\eps_{t,\bx}
\,-\,
\mu^\eps_{t,\by}
\right\|_{L^1(\R^2)}
\,\leq\,
C\,
e^{Ct}\,
\gamma(\bx,\by),\,\,\,
\forall\,
(\bx\,,\, \by) \in K^2\,,
\]
for some positive constant $C$ which may depend on $\eps$ and where $\gamma$ is the following continuous function
\[
\gamma(\bx,\by)
\,=\,
\left\|
\mu^\eps_{0,\bx}
\,-\,
\mu^\eps_{0,\by}
\right\|_{L^1(\R^2)}^{1/2}
\,+\,
\left|
\rho_0^\eps(\bx)
\,-\,
\rho_0^\eps(\by)
\right|
\,+\,
\left\|
\Psi(\bx,\,\cdot)
\,-\,
\Psi(\by,\,\cdot)
\right\|_{L^1(K)}\,.
\]
\end{proposition}

\begin{proof}
Our strategy consists in estimating 
$\ds
H_{1/2}
$ {instead of the $L^1$ error in order to take advantage of the entropy dissipation, then Lemma \ref{entropy V.S. L1} will ensure the continuity in $L^1$.}

We choose some $\bx$ and $\by$ in $K$ all along this proof and consider 
$
\ds
\kappa
\,=\,
\left(
\mu^\eps_{t,\bx}
\,+\,
\mu^\eps_{t,\by}
\right)/2$. 
It satisfies the following equation
\[
\partial_t\, \kappa
        \,+\,
        \partial_v 
        \left( \,
        \left(N(v)
         -w
         \right)\,
         \kappa
         \,\right) 
         \,+\,
         \partial_w
        \left( \,
       A(\bu)\,  \kappa
       \,\right) 
        \,-\,
        \partial_v^2\, 
        \kappa
        \,-\,
        \frac{1}{2}
        \partial_v
        \left(\,
        \mathcal{N}_{
        \bx}\,\mu^\eps_{t,\bx}
        \,+\,
        \mathcal{N}_{
        \by}\,\mu^\eps_{t,\by}
        \right)
        \,=\,0\,,
\]
where $\mathcal{N}$ gathers the non-linear terms and is given by
\[
\mathcal{N}_{
        \bx}(t,v)
        \,:=\,
\frac{\rho^\eps_0(\bx)}{\eps} (v-\cV^\eps(t,\,\bx))
\,+\,
\left(
\Psi*_r\rho_0^\eps(\bx)\,v
\,-\,
\Psi*_r
\left(\rho_0^\eps
\mathcal{V}^\eps
\right)(t,\,\bx)
\right)\,.
\]
We compute the time derivative of 
$\ds
H_{1/2}
\left[\,
\mu^\eps_{t,\,\bx}\,|\,
\mu^\eps_{t,\,\by}\,
\right]
$
using the former equation, equation \eqref{kinetic:eq} and conservation of mass for \eqref{kinetic:eq}. After an integration by part, all the terms associated to linear transport cancel and we obtain
\[
\frac{\dD}{\dD t}
H_{1/2}
\left[\,
\mu^\eps_{t,\,\bx}\,|\,
\mu^\eps_{t,\,\by}\,
\right]
\,+\,
I_{1/2}
\left[\,
\mu^\eps_{t,\,\bx}\,|\,
\mu^\eps_{t,\,\by}\,
\right]
\,=\,
\cH_1\,,
\]
where $I_{1/2}\left[\mu\,|\,\nu\right]$ is given by \eqref{def:Ialpha} and corresponds to the dissipation due to the second order term  whereas  $\cH_1$ is given by
\[
\cH_1
\,=\,
-\int_{\R^2}\,
\partial_{v}
\left(
\frac{\mu^\eps_{t,\bx}}{\kappa}
\right)\,
\left(\,
\mathcal{N}_{\bx}\,\kappa
\,-\,
\frac{1}{2}
\left(
\,\mathcal{N}_{\bx}
\,\mu^\eps_{t,\bx}
\,+
\,
\mathcal{N}_{\by}\,\mu^\eps_{t,\by}
\right)
\right)\,
\dD \bu\,.
\]
Exact computations yield
\[
\cH_1
=
-\frac{1}{2}
\int_{\R^2}\,
\partial_v
\left(
\ln{
	\left(
	\frac{
		\mu^\eps_{t,\bx}
	}{
		\kappa
	}
	\right)
}\,
\right)
\left(\,
\mathcal{N}_{\bx}\,
\,-\,
\mathcal{N}_{\by}\,
\right)
\,
\frac{\mu^\eps_{t,\by}}{\kappa}\,
\mu^\eps_{t,\bx}
\,
\dD \bu\,.
\]
{Then, we first notice that it holds
\[
\left|\,
\mathcal{V}^\eps
\left(
t,\bx
\right)
\,-\,
\mathcal{V}^\eps
\left(
t,\by
\right)
\,
\right|
\,\leq\,2\,W_1\left(\mu^\eps_{t,\bx},\kappa\right)\,.\]
Furthermore, from Proposition \ref{exp moments}, we get that $\mu^\eps_{t,\bx}$, $\mu^\eps_{t,\by}$ and therefore $\kappa$ have exponential moments, hence we can apply Corollary $2.4$ in \cite{Bolley/Villani}, which ensures
\[
W_1\left(\mu^\eps_{t,\bx},\kappa\right)\,\leq\,
C\,H_0
\left[
\mu^\eps_{t,\bx}\,|\,
\kappa
\right]^{1/2}\,.
\]
Since by definition it holds $\ds H_0
\left[
\mu^\eps_{t,\bx}\,|\,
\kappa
\right]
\,=\,H_{1/2}
\left[
\mu^\eps_{t,\bx}\,|\,
\mu^\eps_{t,\by}
\right]$, we obtain
\[
\left|\,
\mathcal{V}^\eps
\left(
t,\bx
\right)
\,-\,
\mathcal{V}^\eps
\left(
t,\by
\right)
\,
\right|
\,\leq\,
C\,H_{1/2}
\left[
\mu^\eps_{t,\bx}\,|\,
\mu^\eps_{t,\by}
\right]^{1/2}\,,
\]
for some positive constant $C>0$ which may depend on $\eps$. }
Consequently, using assumption \eqref{hyp:rho0} and Corollary \ref{cor:1}, we obtain the following bound for $\cH_1$
\[
\left|
\cH_1
\right|
\,\leq\,
C
\left(
\beta(\bx,\by)
\,+\,
H_{1/2}
\left[
\mu^\eps_{t,\bx}\,|\,
\mu^\eps_{t,\by}
\right]^{1/2}
\right)
\int_{\R^2}
\left|
\partial_v
\ln{
\left(
\frac{
\mu^\eps_{t,\bx}
}{
\kappa
}
\right)
}
\right|\,
(1\,+\,|v|)\,
\frac{\mu^\eps_{t,\by}}{\kappa}\,
\mu^\eps_{t,\bx}
\,
\dD \bu\,,
\]
where $C$ is a positive constant that may depend on $\eps$ and where $\beta$ is given by 
\[
\beta(\bx,\by)
\,=\,
\left|
\rho_0^\eps(\bx)
\,-\,
\rho_0^\eps(\by)
\right|
\,+\,
\left\|
\Psi(\bx,\,\cdot)
\,-\,
\Psi(\by,\,\cdot)
\right\|_{L^1(K)}\,.
\]
Then we apply Young's inequality, Proposition \ref{prop:1} and the bound 
$
\ds
\left|\mu^\eps_{t,\by}\,/\,
\kappa \right|
\,\leq\,
2
$. It yields
\[
\left|
\cH_1
\right|
\,\leq\,
C\,\eta^{-1}
\left(
\beta(\bx,\by)^2
\,+\,
H_{1/2}
\left[
\mu^\eps_{t,\bx}\,|\,
\mu^\eps_{t,\by}
\right]
\right)
\,+\,
\eta\,
I_{1/2}
\left[
\mu^\eps_{t,\bx}\,|\,
\mu^\eps_{t,\by}
\right]\,
,
\]
for all positive $\eta$ and 
where $C$ may depend on $\eps$. Therefore, taking $\eta\,=\,1/2$, we obtain 
\[
\frac{\dD}{\dD t}\,
H_{1/2}
\left[
\mu^\eps_{t,\bx}\,|\,
\mu^\eps_{t,\by}
\right]
\,+\,
\frac{1}{2}\,
I_{1/2}
\left[
\mu^\eps_{t,\bx}\,|\,
\mu^\eps_{t,\by}
\right]
\,\leq\,
C
\left(
\beta(\bx,\by)^2
\,+\,
H_{1/2}
\left[
\mu^\eps_{t,\bx}\,|\,
\mu^\eps_{t,\by}
\right]
\right)\,.
\]
We apply Gronwall's lemma and Lemma \ref{entropy V.S. L1}. It yields
\[
\left\|
\mu^\eps_{t,\bx}
\,-\,
\mu^\eps_{t,\by}
\right\|_{L^1(\R^2)}
\,\leq\,
C\,
e^{Ct}
\left(
\left\|
\mu^\eps_{0,\bx}
\,-\,
\mu^\eps_{0,\by}
\right\|_{L^1(\R^2)}^{1/2}
\,+\,
\beta(\bx,\by)
\right)\,.
\]
\end{proof}

We now turn to continuity with respect to the time variable and split the proof into two steps. First we  prove continuity at time $t=0$ and then deduce continuity for all time $t>0$. 

\begin{proposition}[Continuity at time $t=0$]
  \label{continuity:time0}
Under the assumptions of Theorem \ref{WP mean field eq}, the following estimate holds
\[
\sup_{\bx \in K}\,
\left\|
\,
\mu^\eps_{t,\bx}
\,-\,
\mu^\eps_{0,\bx}
\,
\right\|_{L^1(\R^2)}
\,\leq\,
C\,\sqrt{t}\,,\,\,\,
\forall\,
t \in \R_+\,,
\]
for some positive constant $C$ which may depend on $\eps$.
\end{proposition}

\begin{remark}
It is the only time that we use assumption \eqref{hyp4:f0}.
\end{remark}

\begin{proof}
All along this proof, we set 
$
\ds
\kappa 
\,=\,
\left(
\mu^\eps_{t,\bx}
\,+\,
\mu^\eps_{0,\bx}
\right)/2
$. In order to simplify notation, we also define $B^\eps$ as follows
\[
B^\eps(t,\bx,\bu)
\,=\,
N(v)
\,-\,w \,-\,
\mathcal{K}_{\Psi}
[\rho_0^\eps \,\mu^\eps]
\left(
t,\bx,v
\right)
\,-\,
\frac{1}{\eps}\,
        \rho^\eps_0\,
        (v-\cV^\eps)\,,
\]
and we point out that according to Corollary \ref{cor:1}, assumptions \eqref{hyp2:N} on $N$, \eqref{hyp2:psi} \& \eqref{hyp:rho0}, there exists a positive constant $C>0$ that may depend on $\eps$ such that
\[
\left|
B^\eps(t,\bx,\bu)
\right|
\,\leq\,
C
\left(\,
|\bu|^p
\,+\,
1\,
\right),\quad
\forall\, 
\ds
(t,\bx,\bu)
\in
\R_+
\times
K
\times 
\R^2\,.
\]
We compute the derivative of
$
H_{1/2}
\left[\,
\mu^\eps_{0,\bx}
\,|\,
\mu^\eps_{t,\bx}
\,
\right]
$ using equation \eqref{kinetic:eq}. It yields
\begin{align*}
\frac{\dD }{\dD t}
H_{1/2}
\left[
\mu^\eps_{0,\bx}
\,|\,
\mu^\eps_{t,\bx}
\right]
&\,
=\,
\cH_1
\,+\,
\cH_2
\,+\,
\cH_3\,,
\end{align*}
where
\begin{equation*}
\left\{
\begin{array}{l}
    \displaystyle \cH_{1}\,=\,
    -\frac{1}{2}
\int_{\R^2}
\,
\partial^2_v\,
\mu^\eps_{t,\bx}
\,
\frac{\mu^\eps_{0,\bx}}{\kappa}
\,\dD \bu\,,
\\[1.1em]
\displaystyle \cH_{2}\,=\,
    \frac{1}{2}
\int_{\R^2}
\,
\partial_v
\left[
B^\eps\, \mu^\eps_{t,\bx}
\right]
\frac{\mu^\eps_{0,\bx}}{\kappa}
\,\dD \bu\,,
\\[1.1em]
    \displaystyle \cH_{3}\,=\,
    \frac{1}{2}
\int_{\R^2}
\partial_w
\left[
A \,\mu^\eps_{t,\bx}
\right]
\frac{\mu^\eps_{0,\bx}}{\kappa}
\,\dD \bu\,
    .\\[0,8em]
    \end{array}
    \right.
    \end{equation*}
We start with $\cH_1$. First, we integrate by part and rewrite the term as follows
\[
\cH_1
\,=\,
-\,
I_{1/2}
\left[
\mu^\eps_{0,\bx}
\,|\,
\mu^\eps_{t,\bx}
\right]
\,+\,
\int_{\R^2}
\partial_v
\left(
\ln{
\frac{\mu^\eps_{0,\bx}}{\kappa}
}\,
\right)
\left(
1
\,-\,
\frac{1}{2}
\frac{\mu^\eps_{0,\bx}}{\kappa}
\right)
\partial_v
\left(
\ln{
\mu^\eps_{0,\bx}
}\,
\right)
\mu^\eps_{0,\bx}
\,\dD \bu\,.
\]
Second, we apply the inequality
$
\ds
\left|
\mu^\eps_0\,/\,\kappa 
\right|
\,\leq\,
2
$, Young's inequality and assumption \eqref{hyp4:f0} on $\mu^\eps_0$. It~yields for all positive $\eta$
\[
\cH_1
\,\leq\,
-\,
\left(
1-
\frac{\eta}{2}
\right)\,
I_{1/2}
\left[
\mu^\eps_{0,\bx}
\,|\,
\mu^\eps_{t,\bx}
\right]
\,+\,8\,
\frac{m^\eps}{\eta},
\]
We turn to $\cH_2$. After an integration by part, it rewrites as follows
\[
\displaystyle \cH_{2}\,=\,
    -\frac{1}{2}
\int_{\R^2}
\,B^\eps\,
\partial_v
\left(
\ln{
\frac{\mu^\eps_{0,\bx}}{\kappa}
 }
\right)
\frac{\mu^\eps_{t,\bx}}{\kappa}
\,\mu^\eps_{0,\bx}\,
\,\dD \bu\,.
\]
Then we apply the inequality 
$
\ds
\left|
\mu^\eps\,/\,\kappa
\right|
\leq
2
$,
Young's inequality, the bound on $B^\eps$ and the first assumption~in~\eqref{hyp3:f0}. We obtain for some positive constant $C$ and 
for all positive $\eta$
\[
\cH_2
\,\leq\,
\eta\,
I_{1/2}
\left[
\mu^\eps_{0,\bx}
\,|\,
\mu^\eps_{t,\bx}
\right]
\,+\,
\frac{C}{\eta}\,.
\]
To end with, we estimate $\cH_3$. This term is a little trickier to estimate since we do not have dissipation with respect to the adaptation variable. We rewrite 
$\ds
\cH_3\,
=\,
\cH_{31}\,+\,
\cH_{32}$,
where
\begin{equation*}
\left.\right.\\[0,3em]
\left\{
\begin{array}{l}
    \displaystyle \cH_{31}\,=\,
    \int_{\R^2}
\partial_w
\left(
A(\bu)
\right)\,
\psi
\left(
\frac{\mu^\eps_{0,\bx}}{\kappa}
\right)
\frac{\mu^\eps_{0,\bx}}{\kappa}\,
\kappa
\,\dD \bu\,,
\\[1.2em]
\displaystyle \cH_{32}\,=\,
    -\int_{\R^2}
    A(\bu)\,
    \psi\left(
    \frac{\mu^\eps_{0,\bx}}{\kappa}
    \right)\,
    \partial_w
\left(
    \ln{
\mu^\eps_{0,\bx}}
\right)\,
\mu^\eps_{0,\bx}
\,\dD \bu\,,
    \end{array}
    \right.
    \end{equation*}
where 
$\ds\psi(x)\,=\,\ln{(x)}-x/2$.
We notice that 
$\mu^\eps_{0,\bx}/\kappa
$ lies in $[0,\,2]$ and that 
$\ds
\left(x \mapsto x\,\psi(x)
\right)$ is bounded on $[0,\,2]$. Hence, using conservation of mass for equation \eqref{kinetic:eq}, we obtain
\[
\cH_{31}\,\leq\,
    b\sup_{x \in [0,2]}
    \left|
    x\,\psi(x)
    \right|\,.
\]
We turn to $\cH_{32}$. We apply Young's inequality and use assumption \eqref{hyp4:f0} on $\mu^\eps_0$. It yields
\[
\cH_{32}\,=\,
    \frac{1}{2}\int_{\R^2}
    \left|A(\bu)
    \right|^2
    \,
    \left|\,
    \psi\left(
    \frac{\mu^\eps_{0,\bx}}{\kappa}
    \right)
    \right|^2\,
    \frac{\mu^\eps_{0,\bx}}{\kappa}\,
\kappa
\,\dD \bu
\,+\,
2\,m^\eps\,.
\]
We notice that 
$
\ds
\left(
x \mapsto 
x\,
\left|
\psi(x)
\right|^2\right)
$ is bounded on 
$
\ds
[0,2]
$ and that  $\mu^\eps_{0,\bx}/\kappa$ lies in $[0,\,2]$. Furthermore, we apply Proposition \ref{prop:1} and assumption \eqref{hyp1:f0}. It yields
\[
\cH_3\,\leq\,C\,.
\]
Gathering former computations and taking $\eta$ small enough, we obtain 
\begin{align*}
\frac{\dD }{\dD t}
H_{1/2}
\left[
\mu^\eps_{0,\bx}
\,|\,
\mu^\eps_{t,\bx}
\right]
\,+\,
\frac{1}{2}\,
I_{1/2}
\left[
\mu^\eps_{0,\bx}
\,|\,
\mu^\eps_{t,\bx}
\right]
&\,
\leq\,
C\,.
\end{align*}
We integrate the former relation between $0$ and $t$ and apply Lemma \ref{entropy V.S. L1}. Since the constant $C$ does not depend on $\bx$, we take the supremum over $K$. In the end, we obtain
\[
\sup_{\bx \in K}\,
\left\|
\,
\mu^\eps_{t,\bx}
\,-\,
\mu^\eps_{0,\bx}
\,
\right\|_{L^1(\R^2)}
\,\leq\,
C\,\sqrt{t}\,.
\]
\end{proof}
Using Proposition \ref{continuity:time0}, we deduce strong continuity at all times for $\mu^\eps$.
\begin{proposition}[Continuity at time $t>0$]
  \label{continuity:time}
  For any $\varepsilon > 0$, suppose that assumptions \eqref{hyp1:N} and \eqref{hyp1:psi}-\eqref{hyp:rho0} are fulfilled whereas the initial condition $\mu_0^\varepsilon$  verifies the first condition of \eqref{hyp3:f0}, the following estimate holds
\[
\sup_{\bx \in K}
\left\|\,
\mu^\eps_{t\,,\,\bx}
\,-\,
\mu^\eps_{t+h\,,\,\bx}\,
\right\|_{L^1
\left(
\R^2
\right)
}
\,\leq\,
e^{Ct}\,
\sup_{\bx \in K}
\left\|\,
\mu^\eps_{0\,,\,\bx}
\,-\,
\mu^\eps_{h\,,\,\bx}\,
\right\|_{L^1
\left(
\R^2
\right)
}^{1/2},\,\,\,
\forall\,
(t\,,\,h) \in 
\left(
\R_+
\right)^2\,,
\]
for some positive constant $C$ which may depend on $\eps$.
\end{proposition}
\begin{proof}
All along this proof, we consider some $\bx \in K$ and $h>0$. We introduce the following notation
$
\ds
\kappa
\,=\,
\left(
\mu^\eps_{t\,,\,\bx}
\,+\,
\mu^\eps_{t+h\,,\,\bx}
\right)/2$, which satisfies the following equation
\begin{align*}
\partial_t\, \kappa
        \,=\,&-\,
        \partial_v 
        \left( \,
        \left(N(v)
         -w
        - \left(\frac{\rho_0^\eps}{\eps}
         +
         \Psi*_r\rho_0^\eps
         \right)v
         \right)\,
         \kappa
         \,\right) 
         \,-\,
         \partial_w
        \left( \,
       A(\bu)\,  \kappa
       \,\right) 
        \,+\,
        \partial_v^2\, 
        \kappa\\[0,8em]
        &\,-\,
        \frac{1}{2}
        \partial_v
        \left(\,
        \mathcal{N}_{
        t}\,\mu^\eps_{t,\bx}
        \,+\,
        \mathcal{N}_{
        t+h}\,\mu^\eps_{t+h,\bx}
        \right)\,,
\end{align*}
where $\mathcal{N}$ gathers the non-linear terms with respect to the time variable and is given by
\[
\mathcal{N}_{t}(\bx)
        =
\frac{\rho^\eps_0(\bx)}{\eps}\, \cV^\eps(t,\,\bx)
\,+\,
\Psi*_r
\left(\rho_0^\eps
\mathcal{V}^\eps
\right)(t,\,\bx)\,.
\]
We compute the derivative of 
$\ds
H_{1/2}
\left[\,
\mu^\eps_{t,\,\bx}\,|\,
\mu^\eps_{t+h,\,\bx}\,
\right]
$
using the former equation, equation \eqref{kinetic:eq} and conservation of mass for \eqref{kinetic:eq}. After an integration by part, all the terms associated to linear transport cancel and we obtain
\[
\frac{\dD}{\dD t}
H_{1/2}
\left[\,
\mu^\eps_{t,\,\bx}\,|\,
\mu^\eps_{t+h,\,\bx}\,
\right]
\,+\,
I_{1/2}
\left[\,
\mu^\eps_{t,\,\bx}\,|\,
\mu^\eps_{t+h,\,\bx}\,
\right]
\,=\,
\cH_1\,,
\]
where $\cH_1$ is given by
\[
\cH_1
\,=\,
\left(
\mathcal{N}_{t}
\,-\,
\mathcal{N}_{t\,+\,h}
\right)\,
\int_{\R^2}
\partial_v
\left(
\ln{
\frac{\mu^\eps_{t\,,\,\bx}}{\kappa}
}\,
\right)
\frac{\mu^\eps_{t\,,\,\bx}
\,
\mu^\eps_{t+h\,,\,\bx}
}{\kappa}\,
\dD \bu\,.
\]
We use Young's inequality, assumptions \eqref{hyp2:psi} \& \eqref{hyp:rho0} and the inequality 
$
\ds
\left|\,\mu^\eps\,/\,
\kappa \,\right|
\,\leq\,
2
$. It yields
\[
\left|
\cH_1
\right|
\,\leq\,
\eta\,
I_{1/2}
\left[\,
\mu^\eps_{t,\,\bx}\,|\,
\mu^\eps_{t+h,\,\bx}\,
\right]
\,+\,
C\eta^{-1}\,
\sup_{\bx \in K}
\left|
\cV(t,\bx)
\,-\,
\cV(t+h,\bx)
\right|^2\,,
\]
for all positive $\eta$ and some constant $C$ that may depend on $\eps$.
Then from Proposition \ref{exp moments}, we get exponentional moments on $\kappa$ and  apply Corollary $2.4$ in \cite{Bolley/Villani}, which yields
\[
\sup_{\bx \in K}
\left|
\cV(t,\bx)
\,-\,
\cV(t+h,\bx)
\right|^2
\,\leq\,
C\,
\sup_{\bx \in K}
H_{1/2}
\left[\,
\mu^\eps_{t,\,\bx}\,|\,
\mu^\eps_{t+h,\,\bx}\,
\right]\,,
\]
for some constant $C\,>\,0$ which may depend on $\eps$. Gathering the former computations and taking $\eta\,=\,1/2$, it yields
\[
\frac{\dD}{\dD t}
H_{1/2}
\left[\,
\mu^\eps_{t,\,\bx}\,|\,
\mu^\eps_{t+h,\,\bx}\,
\right]
\,+\,
\frac{1}{2}\,
I_{1/2}
\left[\,
\mu^\eps_{t,\,\bx}\,|\,
\mu^\eps_{t+h,\,\bx}\,
\right]
\,\leq\,
C\,
\sup_{\bx \in K}
H_{1/2}
\left[\,
\mu^\eps_{t,\,\bx}\,|\,
\mu^\eps_{t+h,\,\bx}\,
\right]\,.
\]
We integrate this relation between $0$ and $t$ and take the supremum over all $\bx$ in $K$. It yields
\[
\sup_{\bx \in K}
H_{1/2}
\left[\,
\mu^\eps_{t\,,\,\bx}\,|\,
\mu^\eps_{t+h\,,\,\bx}\,
\right]
\,\leq\,
\sup_{\bx \in K}
H_{1/2}
\left[\,
\mu^\eps_{0\,,\,\bx}\,|\,
\mu^\eps_{h\,,\,\bx}\,
\right]
\,+\,
C\,
\int_0^t\,
\sup_{\bx \in K}
H_{1/2}
\left[\,
\mu^\eps_{s\,,\,\bx}\,|\,
\mu^\eps_{s+h\,,\,\bx}\,
\right]\dD s\,.
\]
We apply Gronwall's lemma to the former inequality and use Lemma \ref{entropy V.S. L1}. We obtain
\[
\sup_{\bx \in K}
\left\|
\mu^\eps_{t\,,\,\bx}
\,-\,
\mu^\eps_{t+h\,,\,\bx}
\right\|_{L^1
\left(
\R^2
\right)
}
\,\leq\,
C
e^{Ct}\,
\sup_{\bx \in K}
\left\|
\mu^\eps_{0\,,\,\bx}
\,-\,
\mu^\eps_{h\,,\,\bx}
\right\|_{L^1
\left(
\R^2
\right)
}^{1/2}\,.
\]
\end{proof}
\subsection{Uniqueness}
\label{uniqueness}
We turn to the proof of uniqueness for equation \eqref{kinetic:eq}. We follow the same method as in the proof of Proposition \ref{continuity:time}. We consider two solutions $\ds\mu^1$ and $\ds\mu^2$ to equation \eqref{kinetic:eq} in the sense of Definition \ref{notion de solution} and with the same initial condition $\ds\mu^\eps_{0}$. Then we take some $T>0$ and prove that the following relative entropy
\[
H_{1/2}
\left[\,
\mu^1_{t,\,\bx}\,|\,
\,\mu^2_{t,\bx}
\,
\right],
\]
is zero for all $(t,\bx) \in [0\,,\,T] \times K$. All along this proof, we take some $\bx \in K$ and omit the dependence with respect to $\bx$ and $t$ when the context is clear. Furthermore we write
\[
\left(
\mathcal{V}^1\,,\,
\mathcal{V}^2
\right)
\,=\,
\left(
\int_{\R^2}\,v\,\mu^1\,\dD \bu\,,\,
\int_{\R^2}\,v\,\mu^2\,\dD \bu
\right),
\]
and 
$\ds 
\kappa
\,=\,
\left(\mu^1\,+\,\mu^2
\right)/2
$. Since $\mu^1$ and $\mu^2$ are solution to \eqref{kinetic:eq}, $\kappa$ solves the following equation
\[
\partial_t \kappa
\,+\,
\partial_v 
\left( 
\left(N(v)
-w
-
\left(
\frac{\rho_0^\eps}{\eps}
+
\Psi *_r \rho_0^\eps
\right)
\right)
\kappa
\right) 
\,+\,
\partial_w
\left( 
A(\bu)\, \kappa
\,\right) 
\,-\,
\partial_v^2\, 
\kappa
\,+\,
\frac{1}{2}
\partial_v
\left(\,
\mathcal{N}^1\,\mu^1
+
\mathcal{N}^2\,\mu^2
\right)
=0\,,
\]
where $\mathcal{N}^{i}$, for
$
i \in \{1\,,\,2\}
$
,
 gathers the non-linear terms and is given by
\[
\mathcal{N}^i(t,\,\bx)
=
\frac{\rho^\eps_0(\bx)}{\eps}
\cV^i(t,\,\bx)
\,+\,
\Psi*_r
\left(\rho_0^\eps
\mathcal{V}^i
\right)(t,\,\bx)\,.
\]
We compute the derivative of 
$
H_{1/2}
\left[\,
\mu^1\,|\,
\mu^2\,
\right]
$
using the former equation, equation \eqref{kinetic:eq} and conservation of mass for \eqref{kinetic:eq}. After an integration by part, all the terms associated to linear transport cancel and we obtain
\[
\frac{\dD}{\dD t}
H_{1/2}
\left[\,
\mu^1\,|\,
\mu^2\,
\right]
\,+\,
I_{1/2}
\left[\,
\mu^1\,|\,
\mu^2\,
\right]
\,=\,
\cH_1\,,
\]
where $\cH_1$ is given by
\[
\cH_1
\,=\,
\frac{
\mathcal{N}^1
\,-\,
\mathcal{N}^2
}{2}
\,
\int_{\R^2}
\,
\partial_v
\left(
\ln{
\left(
\frac{\mu^1}{\kappa}
\right)
}
\right)\,
\frac{\mu^2}{\kappa}\,
\mu^1\,
\dD \bu\,.
\]
According to Theorem \ref{WP mean field eq}, $\mu^1$ and $\mu^2$ both have uniformly bounded exponential moments on~$[0,T]$, hence applying Corollary $2.4$ in \cite{Bolley/Villani}, we obtain that
\[
\left|
\mathcal{V}^1
\,-\,
\mathcal{V}^2
\right|
\,\leq\, W_1
\left(\,\mu^1\,,\,\mu^2\,
\right)\,\leq\,
C\,
H_{1/2}
\left[\,
\mu^1\,|\,
\mu^2\,
\right]^{1/2}\,,
\]
for some positive constant $C>0$ that may depend on $T$.
Consequently, we use Young's inequality, the estimate 
$
\ds
\left|
\mu^2/
\kappa
\right|\,\leq\, 2
$,
assumptions \eqref{hyp2:psi} \& \eqref{hyp:rho0} and we obtain
\[
\frac{\dD}{\dD t}
H_{1/2}
\left[\,
\mu^1_{t\,,\,\bx}\,|\,
\mu^2_{t\,,\,\bx}\,
\right]
\,+\,
\frac{1}{2}\,
I_{1/2}
\left[\,
\mu^1_{t\,,\,\bx}\,|\,
\mu^2_{t\,,\,\bx}\,
\right]
\,\leq\,
C\,
\sup_{\bx \in K}
\left(
H_{1/2}
\left[\,
\mu^1_{t\,,\,\bx}\,|\,
\mu^2_{t\,,\,\bx}\,
\right]
\right)\,.
\]
We integrate the former relation between $0$ and $t$ and take the supremum over all $\bx$ in $K$ in the left-hand side. It yields
\[
\sup_{\bx \in K}
\left(
H_{1/2}
\left[\,
\mu^1_{t\,,\,\bx}\,|\,
\mu^2_{t\,,\,\bx}\,
\right]
\right)
\,\leq\,
C\,
\int_0^t
\sup_{\bx \in K}
\left(
H_{1/2}
\left[\,
\mu^1_{s\,,\,\bx}\,|\,
\mu^2_{s\,,\,\bx}\,
\right]
\right)\,\dD s\,.
\]
We apply Gronwall's lemma and conclude the proof.
\subsection{Existence}
\label{section:existence}
In this section, we outline the main ideas in order to construct the solution to equation \eqref{kinetic:eq} given by Theorem \ref{WP mean field eq}. Let us first point out that the continuity of the macroscopic quantities 
$\mathcal{V}^\eps$ and  $\mathcal{W}^\eps$
may be deduced from the continuity and the exponential moments of the solution $\mu^\eps$. Indeed, according to \cite{Bolley/Villani} (see Corollary $2.4$), we have
\[
\left(1-\alpha\right)
\left|
\mathcal{V}^\eps(t,\,\bx)
\,-\,
\mathcal{V}^\eps(s,\,\by)
\right|
\,\leq\,
C
\,
H_{1/2}
\left[
\mu^\eps_{t,\bx},\,
\mu^\eps_{s,\by}
\right]^{{1}/{2}}\,,
\]
as soon as $\mu^\eps_{t,\bx}$ and $
\mu^\eps_{s,\by}$ both have exponential moments. Then we apply Lemma \ref{entropy V.S. L1} and deduce continuity.\\
Hence, it is sufficient to prove that $\mu^\eps$ exists in order to complete the proof. For that matter, we consider the following regularized equation 
\begin{equation}\label{regularized eq}
  \ds\partial_t \, \mu^R
        \,+\,
        \partial_v 
        \left( 
        \left(N^R(v)
         -w
        -\frac{\rho^\eps_0}{\eps} (v-\cV^R)
        -\mathcal{K}_{\Psi}
        \left[\rho_0^\eps \,\mu^R
        \right]\right) \,\mu^R\right) 
        \,+\,
        \partial_w
        \left( 
       A(\bu) \, \mu^R\right) 
        \,-\,
        \partial_v^2\,
        \mu^R\,=\,0\,,
\end{equation}
where
\[
\cV^R(t,\bx) 
        \,=\,
        \ds\int_{\R^2}v~\mu^R_{t\,,\,\bx}(\bu)\,\dD\bu\,,
\]
and where 
$
\ds
\left(
N^R
\right)_{R>0}
$ is a suitable sequence of {globally Lipschitz functions}. We construct solutions to \eqref{regularized eq} with an iterative scheme. In order to prove that the scheme converges, we use the exact same method as in the proof for uniqueness (see Section \ref{uniqueness}).\\
Then we let the truncation parameter $R$ grow to infinity. Making use of the continuity estimates in Section \ref{estimates wp}, we check that Ascoli theorem applies here and prove that the sequence 
$\left(\mu^R\right)_{R\,>\,0}$ is relatively compact. Furthermore, we prove that the limit of any subsequence of $\left(
\mu^R\right)_{R\,>\,0}$ which converges in 
$
\scC^0
\left(
[0,T]\times K\,,\,
L^1
\left(
\R^2
\right)
\right)
$ and which has uniformly bounded exponential moments is a solution to \eqref{kinetic:eq}.


\section{Proof of Proposition \ref{WP combined problem}}\label{proof wp combined pb}
We provide an entropy estimate in order to ensure existence and a uniqueness estimate in order to ensure that
\[
\pi^\eps(t,\cdot)
\in \Pi
\left(
\nu^\eps_{t,\bx},\, 
\mathcal{M}_{\rho_0^\eps}\otimes \bar{\nu}_{t,\bx}
\right),\quad
\forall\, t \in \R^+\,.
\]
We start with the uniqueness result.

\begin{lemma}\label{uniqueness combined pb}
Under the assumptions of Theorem \ref{WP mean field eq},
consider $\bx \in K$, $\eps>0$ and a solution $\pi^\eps$ to equation \eqref{combined problem} in the sense of Definition \ref{notion de solution combined problem} and with initial condition 
$\pi^\eps_0$ lying in 
$\Pi
\left(
\nu^\eps_{0,\bx},\, 
\mathcal{M}_{\rho_0^\eps}\otimes \bar{\nu}_{0,\bx}
\right)
$. Then we have
\[
\pi^\eps(t,\cdot)
\in \Pi
\left(
\nu^\eps_{t,\bx}, \,
\mathcal{M}_{\rho_0^\eps}\otimes \bar{\nu}_{t,\bx}
\right),\quad
\forall\, t \in \R^+\,.
\]
\end{lemma}

\begin{proof}
We define $\pi^1$ (resp. $\pi^2$) the marginal of $\pi^\eps$ with respect to $\bu$ (resp. $\bu'$) and we drop the dependence with respect to time and space when the context is clear.
We integrate equation \eqref{combined problem} with respect to $\bu'$. We obtain that
the difference between $\pi^1$ and $\nu^\eps_{t,\,\bx}$ solves
\begin{equation*}
\partial_t\, 
\left(\pi^1
-
\nu^\eps
\right)
\,+\,
\text{div}_{\bu}
\left[
\mathbf{b}^\eps_0
\,
\left(\pi^1
-
\nu^\eps
\right)
\right]
\,=\,
\frac{1}{\eps}\,
\partial_v
\left[\,\rho_0^\eps
\,v\, 
\left(\pi^1
-
\nu^\eps
\right)
+
\partial_v \,
\left(\pi^1
-
\nu^\eps
\right)
\right]\,.
\end{equation*}
Multiplying the former equation by 
$
\mathrm{sign}
\left(
\pi^1
-
\nu^\eps
\right)
$, we obtain that
$
\left|
\pi^1
-
\nu^\eps
\right|
$
is a sub-solution to the former equation. Then we integrate with respect to $\bu$ and obtain
\[
\frac{\dD}{\dD t}
\left\|
\pi^1
\,-\,
\nu^\eps
\right\|_{L^1\left(\R^2\right)}
\,\leq\,
0\,.
\]
We follow the method same for 
$
\ds
\pi^2
$
and obtain the expected result.
\end{proof}

We end this section with an entropy estimate, which ensures existence for solutions to \eqref{combined problem} in the sense of Definition \ref{notion de solution combined problem}. We define the Fisher information associated to \eqref{combined problem}
\[
I_{v,v'}
\left[\,
\pi\,
\right]
\,=\,
\int_{\R^4}
\left|\,
\left(\,
\partial_v
\,+\,
\partial_{v'}\,
\right)
\ln{\pi}\,
\right|^2\,
\pi\,\dD\bu\,\dD \bu'\,.
\]
\begin{lemma}
Under the assumptions of Theorem \ref{WP mean field eq},
consider $\bx \in K$, $\eps>0$ and
a solution $\pi^\eps$ to equation \eqref{combined problem} with some initial condition $\pi^\eps_0$ lying in 
$
\ds
\Pi
\left(
\nu^\eps_{0,\bx}, 
\mathcal{M}_{\rho_0^\eps}\otimes \bar{\nu}_{0,\bx}
\right)
$, there exists a positive constant $C>0$ which may depend on $\eps$, $m_*$, $m_p$ and $\ols{m}_p$ such that
\[
H
\left[\,
\pi^\eps_t\,
\right]
\,\leq\,
H
\left[\,
\pi^\eps_0\,
\right]
\,+\,
C\,t\,.
\]
\end{lemma}

\begin{proof}
We integrate equation \eqref{combined problem} with respect to $\bu$ and $\bu'$ and deduce that mass is conserved through time. Hence,
we compute the derivative of 
$
\ds
H
\left[\,
\pi^\eps_t\,
\right]
$
multiplying equation \eqref{combined problem}
by 
$\ds
\ln{
\left(
\pi^\eps_t
\right)}$. After an integration by part, it yields
\[
\frac{\dD}{\dD t}
H
\left[\,
\pi^\eps_t\,
\right]
\,+\,
\frac{1}{\eps}\,
I_{v,v'}
\left[\,
\pi^\eps_t\,
\right]
\,=\,
\Psi*_r\rho_0^\eps(\bx)
\,+\,
2
\left(b
\,+\,
\frac{\rho_0^\eps}{\eps}
\right)
\,+\,
\cA\,,
\]
where $\cA$ is given by
\[
\cA
\,=\,
\frac{1}{\sqrt{\eps}}
\int_{\R^4}
N
\left(
\mathcal{V}^\eps
+
\sqrt{\eps}\,v
\right)\,
\partial_v
\left(
\ln{
\pi^\eps_t
}
\right)\,
\pi^\eps_t\,
\dD \bu \,\dD \bu'\,.
\]
Since the Fisher information $I_{v,v'}$ is somehow degenerate, we can not apply directly Young's inequality to $\cA$. Hence, we re-write $\cA$ as follows
\[
\cA
\,=\,
\frac{1}{\sqrt{\eps}}
\int_{\R^4}
N
\left(
\mathcal{V}^\eps
+
\sqrt{\eps}\,v
\right)\,
\left(
\,
\partial_v
\,+\,
\partial_v'
\,\right)\,
\left(
\ln{
\pi^\eps_t
}
\right)\,
\pi^\eps_t\,
\dD \bu \,\dD \bu'\,.
\]
Then, we apply Young's inequality to $\cA$ and obtain
\[
\cA
\,\leq\,
\frac{C\eta}{\eps}\,
I_{v,v'}
\left[\,
\pi^\eps_t\,
\right]
\,+\,
C\eta\,
\int_{\R^4}
\left|
N
\left(
\mathcal{V}^\eps
+
\sqrt{\eps}\,v
\right)
\right|^2\,
\pi^\eps_t\,
\dD \bu \,\dD \bu'\,,
\]
for all $\eta$ in $]0,1[$. Applying Lemma \ref{uniqueness combined pb}, assumption \eqref{hyp2:N}, Proposition \ref{prop:2} and Corollary \ref{cor:1}, it yields
\[
\int_{\R^4}
\left|
N
\left(
\mathcal{V}^\eps
+
\sqrt{\eps}\,v
\right)
\right|^2\,
\pi^\eps_t\,
\dD \bu \,\dD \bu'
\,\leq\,C\,.
\]
Hence, taking $\eta$ small enough, we obtain 
\[
\frac{\dD}{\dD t}
H
\left[\,
\pi^\eps_t\,
\right]
\,+\,
\frac{1}{2\,\eps}\,
I_{v,v'}
\left[\,
\pi^\eps_t\,
\right]
\,\leq\, C\,,
\]
where $C$ may depend on $\eps$, $m_p$ and $\ols{m}_p$. We integrate the former inequality between $0$ and $t$ and obtain the result. 
\end{proof}


\end{document}